\documentclass[a4paper,reqno]{amsart}

\usepackage{amsmath,amsthm,amssymb,amsfonts,mathrsfs}
\usepackage{esint}
\usepackage[usenames]{color}

\theoremstyle{plain}
\newtheorem{theorem}{Theorem}[section]
\newtheorem{corollary}[theorem]{Corollary}

\newtheorem{lemma}[theorem]{Lemma}
\newtheorem{proposition}[theorem]{Proposition}
\theoremstyle{definition}

\theoremstyle{remark}
\newtheorem{remark}{Remark}[section]
\numberwithin{equation}{section}

\newcommand{\RR}{\mathbb{R}}
\newcommand{\Rd}{\mathbb{R}^{d\times d}}
\newcommand{\Rds}{\mathbb{R}^{d\times d}_{{\rm sym}}}
\newcommand{\Rdsk}{\mathbb{R}^{d\times d}_{{\rm skw}}}
\newcommand{\Sph}{\mathbb{S}^2}

\newcommand{\tr}{{\rm tr}}
\newcommand{\dist}{{\rm dist}}

\newcommand{\Wvol}{W_{{\rm vol}}}
\newcommand{\ep}{\varepsilon}
\newcommand{\vph}{\varphi}
\newcommand{\Om}{\Omega}
\newcommand{\Wph}{W_h^{1,p}}
\newcommand{\Wp}{W^{1,p}(\Omega,\mathbb R^d)}
\newcommand{\pa}{\partial}
\newcommand{\na}{\nabla}

\begin{document}

\title[$\Gamma$-limits of multiwell energies]
{From nonlinear to linearized elasticity via $\Gamma$-convergence:
the case of multiwell energies satisfying weak coercivity conditions}
\author[V.~Agostiniani]{Virginia Agostiniani}
\address{Mathematical Institute\\
University of Oxford\\
24-29 St. Giles'\\
Oxford, OX1 3LB, UK
}
\email{Virginia.Agostiniani@maths.ox.ac.uk}

\author[T.~Blass]{Timothy Blass}
\address{
Center for Nonlinear Analysis\\
Department of Mathematical Sciences\\
Carnegie Mellon University\\
Pittsburgh, PA 15213, USA 
}
\email{tblass@andrew.cmu.edu}

\author[K.~Koumatos]{Konstantinos Koumatos}
\address{Mathematical Institute\\
University of Oxford\\
24-29 St. Giles'\\
Oxford, OX1 3LB, UK
}
\email{koumatos@maths.ox.ac.uk}

\thanks{V.~A. and K.~K. have received funding from the 
European Research Council under the 
European Union's Seventh Framework Programme (FP7/2007-2013) / ERC grant agreement ${\rm n^o}$ 291053.
T.~B. was supported by the NSF under the
PIRE Grant No. OISE-0967140. The authors would like to thank 
A. DeSimone for his contribution leading to the formulation of 
the problem and J. Kristensen for useful discussions about the 
Young measure representation.}

\begin{abstract} 
Linearized elasticity models are derived, via $\Gamma$-convergence, 
from suitably rescaled nonlinear energies when the corresponding energy 
densities have a multiwell structure and satisfy a weak coercivity 
condition, in the sense that the typical quadratic bound from below 
is replaced by a weaker $p$ bound, $1<p <2$, away from the wells. 
This study is motivated by, and our results are applied to, energies arising in 
the modeling of nematic elastomers.
\end{abstract}
\maketitle

\textsc{MSC (2010): 74G65, 49J45, 74B15, 76A15}

\keywords{Linearized elasticity, $\Gamma$-convergence, nematic elastomers}

\section{Introduction} 

Consider a homogeneous and hyperelastic body occupying in its 
reference configuration a bounded domain $\Om\subset\RR^d$. 
Deformations of the body are described by mappings $v:\Om\to\RR^d$, where 
$v(x)$ denotes the deformed position of the material point $x\in\Om$. 
The total elastic energy
corresponding to the deformation $v$ is given by
\[
\int_\Om W(\na v(x))dx,
\]
where $\na v\in\Rd$ is the deformation gradient and $W:\Rd\to\RR$ is 
a frame-indifferent energy density associated with the material. 
More generally, we consider energies of the form
\[
\int_\Om W(\na v(x))\,dx - \int_\Om l(x)\cdot u(x)\,dx,
\]
where $u(x)=v(x)-x$ is the displacement and $l(x)$ is an external 
(dead) load at $x\in\Om$, so that the term $\int_\Om l\cdot u\,dx$ 
accounts for the work performed by the applied loads.

Let us illustrate the idea behind the passage from
nonlinear to linearized elasticity.
Suppose that $W$ is $C^2$ near the identity, nonnegative 
(up to additive constants), and vanishing precisely on $SO(d)$. 
In the absence of external loads, the deformation $v(x)=x$ is 
an equilibrium state and it is natural to expect that small 
external loads $\ep l$ result in small displacements $\ep u$, 
where $\ep > 0$ is a small parameter. The associated energy 
then becomes
\begin{equation}
\label{eq:intro_energy}
\int_\Om W(I+\ep\na u(x))\,dx - \ep^2\int_\Om l(x)\cdot u(x)\,dx,
\end{equation}
where $I\in\Rd$ is the identity matrix.
 
Assuming that $u\in W^{1,\infty}(\Om,\RR^d)$ and 
rescaling \eqref{eq:intro_energy} by $1/\ep^2$, the limit as 
$\ep\to0$ yields
\begin{equation}
\label{eq:intro_linear}
\frac{1}{2}\int_\Om D^2W(I)[e(u)]^2\,dx - \int_\Om l(x)\cdot u(x)\,dx,
\end{equation}
where $e(u):=[\na u+(\na u)^T]/2$. Recall 
that the quadratic form $M\mapsto D^2W(I)[M]^2$ acts only on the symmetric part of $M$,
due to frame-indifference.

Note that this argument is restricted to $u\in W^{1,\infty}(\Om,\RR^d)$ 
and does not entail whether minimizers $u_\ep$ of the rescaled nonlinear energies
\begin{equation*}
\label{eq:intro_nonlinear}
\frac{1}{\ep^2}\int_\Om W(I+\ep\na u(x))\,dx - \int_\Om l(x)\cdot u(x)\,dx,
\end{equation*}
subject to suitable boundary data, actually converge to the 
minimizer of the limiting linearized\footnote{Note that the limiting energy density in \eqref{eq:intro_linear}
 is quadratic and corresponds to a linear stress--strain relation. 
 For multiwell energies this is not the case and one may only speak 
 of \textit{geometrically linear} models. Thus, the term linearized is 
 preferred over the term linear.} problem \eqref{eq:intro_linear}, 
under the same boundary data.  
The rigorous derivation of the linearized elastic formula \eqref{eq:intro_linear} 
from nonlinear elasticity was provided in~\cite{dal2002linearized} 
via $\Gamma$-convergence, under the condition
\begin{equation}
\label{eq:intro_2growth}
W(F)\geq c\,\dist ^2 (F, SO(d)).
\end{equation}

In this paper, we derive linearized models from nonlinear energies 
with a multiwell structure, i.e.~$W$ is minimized on a set $\mathcal{U}$ 
of the form $SO(d)U$, $U$ ranging in a compact subset of positive definite, symmetric matrices. 
Also, we weaken condition \eqref{eq:intro_2growth} (with $SO(d)$ 
replaced by $\mathcal{U}$) to
\begin{equation*}
\label{eq:intro_pgrowth}
W(F)\geq c\,\dist ^p (F, \mathcal{U}),
\end{equation*}
$1 < p < 2$, for $F$ away from $\mathcal{U}$; the coercivity 
remaining quadratic near $\mathcal{U}$. 

Energies of this type arise naturally in a large class of compressible 
models for rubber-like materials, including nematic elastomers, 
the latter being materials consisting of networks of polymer chains 
with embedded liquid crystalline molecules.
In \cite{DeTe}, some nonlinear compressible models for nematic 
elastomers are considered together with their formally derived 
small-strain theories.
These nonlinear models satisfy our assumptions and our results 
rigorously justify their geometrically linear 
counterparts (see Theorem~\ref{thm:nematic}).

In order to derive small-strain limiting theories, we introduce a 
small parameter $\ep$ and we consider a family of densities $\{W_\ep\}$ 
with corresponding energy wells 
\begin{equation}
\label{eq:intro_wells}
\mathcal U_\ep := SO(d)\{U_\ep = U^T_\ep = I + \ep U + o(\ep) : U\in\mathcal M\}, 
\end{equation}
$\mathcal M$ being a compact subset of symmetric matrices 
and $o(\ep)$ being uniform with respect to $U\in\mathcal M$. 
We assume that
\begin{equation}
\label{eq:intro_ep_growth}
W_\ep\geq c\,\dist ^2 (\cdot, \mathcal U_\ep)\ \mbox{ near }\mathcal U_\ep,\qquad
W_\ep\geq C\,\dist ^p (\cdot, \mathcal U_\ep)\ \mbox{ away from }\mathcal U_\ep,
\end{equation}
and we investigate the limiting behavior, as $\ep\to0$, of the rescaled functionals
\[
\mathscr{E}_\ep(u):=\frac{1}{\ep^2}\int_\Om W_\ep(I+\ep\na u(x))\,dx - \int_\Om l(x)\cdot u(x)\,dx
\]
and their (almost) minimizers. For a discussion on the choice 
of the various scalings, the reader is referred to \cite{Schmidt}.

In view of the coercivity assumption, the natural ambient 
space is $W^{1,p}(\Om,\RR^d)$, where one can prove 
equicoercivity of the functionals $\mathscr E_\ep$. This compactness, 
coupled with a $\Gamma$-convergence result, allows us to 
prove that, under suitable boundary data, the infima 
of $\mathscr E_\ep$ over $W^{1,p}(\Om,\RR^d)$ converge to the infimum of
\[
\mathscr E(u):=\int_\Om V(e(u(x)))dx - \int_\Om l(x)\cdot u(x) dx
\]
over $W^{1,2}(\Om,\RR^d)$, under the same boundary conditions. 
The linearized energy density $V$ is obtained as the limit $\ep\to0$ 
of the quantities $\ep^{-2}W(I+\ep \cdot)$, whenever this limit is 
uniform on compact subsets of symmetric matrices. Moreover, 
sequences of almost minimizers of the functionals $\mathscr E_\ep$ 
converge to a minimizer of the relaxation of $\mathscr E$ in 
$W^{1,p}(\Om,\RR^d)$. This is the content of Theorem~\ref{main_thm}.

We remark that the first attempt to rigorously justify the passage 
from nonlinear to linearized elasticity in the case of multiwell 
energy densities was due to B.~Schmidt in \cite{Schmidt}, 
where the author assumes the standard quadratic coercivity condition 
(corresponding to \eqref{eq:intro_ep_growth} with $p=2$)
\footnote{Note that in \cite{Schmidt} the author 
also assumes that the set $\mathcal M$ appearing in \eqref{eq:intro_wells} 
consists of a finite number of symmetric matrices. However, the same results 
extend to the case of a compact set $\mathcal M$ 
without changing the proofs.}. 
In the same 
paper Schmidt applies his results to discuss the validity of the 
so-called KRS model \cite{khachaturyan_book} for crystalline 
solids which can be thought of as a formal linearization of nonlinear 
theories for solid-to-solid phase transitions (see \cite{bj92} and 
\cite{kohn1991}). The theory developed in \cite{Schmidt} was 
later applied in \cite{AgDe} to justify certain linearized models 
for nematic elastomers with quadratic growth. To include other 
natural compressible models for nematic elastomers, we extend 
the results of \cite{Schmidt} to the case $1 < p < 2$.
Some of the proofs rely on techniques 
introduced in \cite{AgDalDe} where the case of single well 
energies satisfying the weak coercivity condition is treated. 

The paper is organized as follows: in Section \ref{sec:main}, 
we introduce all the ingredients and state our main results. 
The models for nematic elastomers under consideration are 
described in detail in Section~\ref{sec:application} where the 
results of Section~\ref{sec:main} are applied. Section~\ref{sec:cpt} 
is devoted to the proofs of the main statements.
In Section~\ref{sec:equiint} we prove that it is
possible to provide a Young measure representation for the limiting functional, as
well as prove the strong convergence of sequences of almost minimizers under strong 
convexity assumptions on the limiting density $V$.
The paper concludes with an Appendix where various already 
established results are gathered, along with their proofs, for the convenience of the reader. 


\section{Main results}\label{sec:main}
The sets of matrices we work with are $\Rd$ ($d\times d$ real matrices),
$\Rds$ (symmetric matrices), $SO(d)$ (rotations).  
Here and throughout, $c>0$ denotes a generic constant which might differ in each instance. 
We denote by $id$ the identity function on $\RR^d$ and by $I\in\Rd$ the identity matrix.
For every $F\in\Rd$, ${\rm sym}\,F:=\frac{F+F^T}2$.

Let $W_{\ep}:\Rd\rightarrow\overline{\RR}:=\RR\cup\left\{\infty\right\}$ be a 
family of frame-indifferent multiwell energy densities with a corresponding set of wells
$\mathcal U_\ep$ given by \eqref{eq:intro_wells}.
Note that the matrices $U_\ep\in\mathcal U_\ep$ are positive definite
for every $\ep$ small enough.
We also assume that the energies $W_\ep$ are measurable, continuous in an $\ep$-independent 
neighborhood of the identity and satisfy the following coercivity condition:
\begin{equation}
\label{eq:coercivity}
W_\ep(F)\geq c g_p({\rm dist}(F,\mathcal{U}_\ep)),
\end{equation} 
for all $F\in\mathbb{R}^{d\times d}$ and for a constant $c>0$ independent of $\ep$, 
where for some $1<p\leq 2$, $g_p:[0,\infty)\rightarrow\mathbb{R}$ is given by:
\begin{equation}
\label{eq:gp}
g_p(t)=\left\{\begin{array}{ll}\frac{t^2}{2},& t\in[0,1]\\
\,&\,\\
\frac{t^p}{p}+\left(\frac{1}{2}-\frac{1}{p}\right),& t\in[1,\infty).\end{array}\right.
\end{equation}
To retain physicality, though not required for the proofs, we impose the additional condition that
\begin{align*}
\label{eq:growth}
W_{\ep}(F)&\rightarrow +\infty,\quad\mbox{as }\ \det F\rightarrow 0, \\
W_{\ep}(F)&=+\infty,\quad\mbox{if }\ \det F\leq 0.
\end{align*}

The reference configuration is represented by a bounded and Lipschitz domain $\Om\subset\RR^d$ and,
to incorporate the boundary data, we fix $h\in W^{1,\infty}(\Om,\RR^d)$, a subset $\pa_D\Om\subseteq\pa\Om$ of
positive surface measure, and introduce for every $1<p\leq 2$ the set 
\begin{equation}\label{Wph}
\Wph:=\{u\in W^{1,p}(\Om,\RR^d):Tu=Th\mbox{ on }\pa_D\Om\},
\end{equation}
where $T$ stands for the trace operator. We require $\partial_D\Om$
to have \emph{Lipschitz boundary} in $\partial\Om$ according to \cite[Definition 2.1]{AgDalDe}.
This condition implies that $\Wph$ agrees with the closure of $W^{1,\infty}_h$ in 
$W^{1,p}(\Om,\RR^d)$ (see \cite[Proposition A.2]{AgDalDe}). We use this equivalence in 
the proof of Theorem~\ref{thm:gammaconvergence} below.

A continuous and linear functional $\mathscr L:\Wp\to\RR$,
with $p$ as in (\ref{eq:coercivity}), represents the applied loads.
It is in principle a function of the deformation $v$, but
it enters the expression of the total energy
of the system only as $\mathcal L(u)$, where $u(x)=v(x) - x$ is the displacement
associated with the deformation $v$. This is because the total energy can be renormalized
by $-\mathscr L(id)$, in view of the linearity of $\mathscr L$.    

The problem under investigation is to understand the behavior, as $\ep\to0$, 
of the infimum of the total energy appropriately
rescaled by $1/\ep^2$:
$$
\frac{1}{\ep ^2}\int_{\Omega}W_\ep(I+\ep\nabla u)dx-\mathscr L(u)
$$
subject to the boundary data $h$.
The analysis is thus based on the rescaled quantities $W_\ep(I+\ep F)/\ep^2$, whose limit as $\ep\to0$
depends only on the symmetric part of $F$, due to frame indifference. 
Thus, we consider the rescaled densities
\begin{equation}\label{lin_lim_1}
V_\ep(E):=\frac{1}{\ep ^2}W_\ep(I+\ep E),
\end{equation}
defined for every $E\in\Rds$,
or, equivalently, their extensions $f_\ep:\mathbb{R}^{d\times d}\rightarrow\RR$ given by
\begin{equation}\label{eq:fep}
f_\ep(F):=V_\ep({\rm sym}\,F).
\end{equation}
Assume that $V_\ep\rightarrow V$ uniformly on compact subsets of $\Rds$ for some $V:\Rds\rightarrow\RR$. 
Note that this is equivalent to asking that 
$f_\ep\to f$ uniformly on compact subsets of $\Rd$ where $f:\Rd\rightarrow\RR$ is the extension of $V$ given by
\begin{equation}\label{eq:fonly}
f(F):=V({\rm sym}\,F).
\end{equation}
Also, we remark that $V$ satisfies the growth condition
\begin{equation*}
\label{eq:vgrowth}
V(E)\leq c(1+\vert E\vert ^2)
\end{equation*}
if and only if $f(F)\leq c(1+\vert F\vert^2)$. 

Observe that, in view of the growth condition (\ref{eq:coercivity}) and 
Lemma \ref{lem:dist}, if $V(E)=0$ then $E\in\mathcal M$, where $\mathcal M$ 
appears in definition (\ref{eq:intro_wells}).
The following theorem  is our main result.

\begin{theorem}
\label{main_thm}
Let $1<p\leq 2$, suppose that $f_\ep\rightarrow f$ uniformly on compact subsets of $\Rd$, 
and that  $f$ satisfies $0\leq f(F)\leq c(1+\vert F\vert^2)$ for every $F\in\Rd$ and some constant $c>0$. 
If
$$
m_{\ep}:=\inf_{u\in\Wph}\left\{\frac{1}{\ep ^2}\int_{\Omega}W_\ep(I+\ep\nabla u)dx-\mathscr L(u)\right\},
$$
and if $\{u_{\ep}\}$ is a sequence such that
\begin{equation}\label{almost_min}
\lim_{\ep\to0}\left\{\frac{1}{\ep ^2}\int_{\Omega}W_\ep(I+\ep\nabla u_\ep)dx-\mathscr L(u_{\ep})\right\}
=\lim_{\ep\to0}m_{\ep},
\end{equation}
then, up to a subsequence, $u_{\ep}\rightharpoonup u$ in $\Wp$, where $u$ is a solution to the minimum problem
\begin{equation}\label{eq:min_pb}
m:=\min_{u\in W_h^{1,2}}\left\{\int_{\Om}f^{qc}(\na u)dx-\mathscr L(u)\right\}.
\end{equation}
Moreover, $m_{\ep}\to m$.
\end{theorem}

The integrand $f^{qc}$ obtained in the limit is the \emph{quasiconvexification} of $f$.
The corresponding notion for $V$ is the \emph{quasiconvexification on linear strains}
(see Subsection \ref{subsection:quasi} for definitions), and we denote it by $V^{qce}$.
Although we have the equality $f^{qc}(F)=V^{qce}({\rm sym}\,F)$ for every $F\in\Rd$ (see Proposition \ref{prop:qcqce}),
we prefer to retain both the notation $f^{qc}$ and $V^{qce}$, because while our 
proofs seem more natural in terms of $f^{qc}$, some results are more easily stated in terms 
of $V^{qce}$. 

\begin{remark}
Note that condition \eqref{eq:coercivity} and the definition of $V$ as 
the uniform limit on compact subsets of $\Rds$ of $V_\ep$ in \eqref{lin_lim_1},
yields the existence of constants $C_1,C_2 >0$ such that
\begin{equation*}
  V(E) \geq C_1 |E|^2 - C_2.
\end{equation*}
Indeed, for $\ep$ sufficiently small, ${\rm dist} (I + \ep E,\mathcal{U}_\ep)\leq 1$ 
so that by the definition of $g_p$ 
\begin{align*}
  V(E) & \geq 
\limsup_{\ep \to 0}\frac{c}{2\ep^2} {\rm dist}^2(I + \ep E,\mathcal{U}_\ep)\\
&\geq  C_1\lim_{\ep \to 0}\frac{1}{\ep^2}{\rm dist}^2(I+\ep E, SO(d)) -C_2=C_1 |E|^2-C_2,
\end{align*}
where in the last equality we have used \eqref{eq:limdSOd}.
\end{remark}

The proof of Theorem \ref{main_thm} is based on two intermediate results: a compactness result
following from equicoercivity, and a $\Gamma$-convergence result. In order to state them, 
we define the approximate functionals $\mathscr E_{\ep}:\Wp\rightarrow\overline{\RR}$ by
\begin{equation}
\label{eq:Eepsilon}
\mathscr E_{\ep}(u):=
\left\{
\begin{array}{ll}\frac{1}{\ep ^2}\int_{\Omega}W_\ep(I+\ep\nabla u)\,dx,& u\in W^{1,p}_h\\
\,& \,\\
+\infty ,& \mbox{otherwise},\end{array}\right.
\end{equation}
and the limiting functional $\overline{\mathscr E}:\Wp\rightarrow\overline\RR$ by
\begin{equation}
\label{eq:Erel}
\overline{\mathscr E}(u):=\left\{\begin{array}{ll}\int_{\Omega} f^{qc}(\nabla u)\,dx, & u\in W^{1,2}_h\\
\,&\,\\
+\infty ,& \mbox{otherwise.}\end{array}\right.
\end{equation}

\begin{proposition}[Equicoercivity]
\label{prop:compactness}
Let $1\leq p\leq 2$. There exists a constant $C=C(\Om,p,\pa_D\Om,h)>0$ such that
\begin{equation*}
\int_{\Om}|\na u|^pdx\leq C(1+\mathscr E_\ep(u))
\end{equation*}
for every $u\in\Wph$ and every $\ep$ sufficiently small.
\end{proposition}

This result allows us to deduce that, in the case $1<p\leq2$, 
if we have a sequence $\{u_{\ep}\}$ of almost minimizers, that is 
$\{u_{\ep}\}$ satisfies (\ref{almost_min}), then, up to a subsequence, 
$u_{\ep}\rightharpoonup u\in\Wph$. By standard $\Gamma$-convergence arguments, 
Theorem \ref{thm:gammaconvergence} below then implies that $u$ is indeed a solution of the
minimum problem (\ref{eq:min_pb}). 
 
\begin{theorem}[$\Gamma$-convergence]
\label{thm:gammaconvergence}
Let $1<p\leq 2$. Under the hypotheses of Theorem \ref{main_thm}, 
the sequence of functionals $\{\mathscr E_\ep\}$ $\Gamma$-converges to $\overline{\mathscr E}$ 
with respect to the weak topology of $W^{1,p}(\Om,\RR^d)$.
\end{theorem}

To obtain this result, the requirement $h\in W^{1,\infty}(\Om,\RR^d)$ is sharp in the sense that
there are some particular $h\in W^{1,q}(\Om,\RR^d)$, with $2\leq q<\infty$ such that
the $\Gamma$-convergence does not hold unless the energy densities 
satisfy suitable bounds from above which are not natural in this context
(see \cite[Remark 2.7]{AgDalDe}).

\begin{remark}[Relaxation]
\label{rmk:relaxation}
Notice that from Theorem \ref{thm:gammaconvergence} we have also obtained, as a by-product, 
that $\overline{\mathscr E}$ is the sequentially weak lower semicontinuous envelope
in $\Wp$ of the functional $\mathscr E:\Wp\rightarrow\overline\RR$ defined by
\begin{equation*}
\label{eq:Elimit}
\mathscr E(u):=
\left\{
\begin{array}{ll}\int_{\Omega} f(\nabla u)\,dx, & u\in W^{1,2}_h\\
\,&\,\\
+\infty ,& \mbox{otherwise.}\end{array}\right.
\end{equation*}

Indeed, by standard $\Gamma$-convergence results, the functional $\overline{\mathscr E}$ is
weakly lower semicontinuous in $W^{1,p}(\Om,\RR^d)$. Hence, whenever $u_j\rightharpoonup u$ in 
$W^{1,p}(\Om,\RR^d)$, we have that
\[
\overline{\mathscr E}(u)\leq\liminf_{j\to\infty}\overline{\mathscr E}(u_j)\leq
\liminf_{j\to\infty}\mathscr E(u_j),
\]
since $\overline{\mathscr E}\leq\mathscr E$. On the other hand, to prove the existence of a 
relaxing sequence it is enough to consider $u\in W^{1,2}_h$ and then 
standard relaxation results (see Theorem \ref{thm:relaxation})
for the functional $\mathscr E$ restricted to $W^{1,2}(\Om,\RR^d)$ provide the required sequence.
In particular, we have that $\min_{W^{1,2}_h}\overline{\mathscr E}=\inf_{W^{1,2}_h}\mathscr E$.
Note also that $\overline{\mathscr E -\mathscr L} = \overline{\mathscr E} - \mathscr L$.
\end{remark}

Next we present Corollary \ref{cor:lowen}, analogous to \cite[Corollary~2.8]{Schmidt}. We begin by 
introducing some notation. 
Let  $\mathcal Q$ denote the set of quasiconvex functions
from $\Rd$ to $\RR$ and for $1\leq q<\infty$ let $\mathcal Q_q$ denote the set of functions $f\in\mathcal Q$
such that $0\leq f(F)\leq C(1+|F|^q)$ for every $F$ and some $C>0$.
For a compact set $K\in \Rd$, the \emph{strong $q$-quasiconvex hull} of $K$ is
\begin{equation}
  \label{Khulls}
   {\bf Q}_qK: = \left\{F\in \Rd \, : \,f(F)\leq \sup_{G\in K}f(G)
\mbox{ for every }f\in \mathcal Q_q \right\}.
\end{equation}
Recall that the \emph{quasiconvex hull} $QK$ of $K$ is defined as the right-hand side of (\ref{Khulls})
with $\mathcal Q$ in place of $\mathcal Q_q$.
On the other hand, the \emph{weak $q$-quasiconvex hull} $Q_q K$ of $K$ is 
 the zero-level set of the quasiconvexification of the function $F\mapsto{\rm dist}^q(F,K)$. 
Finally, we define the sets
$Q^eK$, ${\bf Q}^e_qK$, and $Q^e_q K$ analogously 
in terms of quasiconvexity on linear strains. 

\begin{corollary}\label{cor:lowen}
Suppose $\partial_D \Omega = \partial \Omega$, and $h(x) = Fx$
for a fixed $F\in \Rd$. If $\{u_\ep \} \subset W^{1,p}_h$ satisfies
\begin{equation*}
  \label{lowen}
  \liminf_{\ep \to 0}\left\{\frac1{\ep^2}\int_{\Om}W_{\ep}(I+\ep\na u_{\ep})dx\right\}=0,
\end{equation*}
then ${\rm sym}\,F \in\{V^{qce}=0\}\subseteq Q^e_2\mathcal M$. 
\end{corollary}
The idea is that at low energy scales,
i.e. $\ep^{-2}\int_{\Omega}W_\ep \ll 1$, 
restrictions are imposed on the possible boundary data $F$ so that they are
compatible with the wells $U\in \mathcal M$. For example,
it is straightforward to show (see the proof of Corollary \ref{cor:lowen})
that the data $F$ must satisfy  ${\rm sym}\, F\in\{V^{qce}=0\}$.
However, 
such restrictions can be improved and 
are typically expressed in terms of some quasiconvex hull of the wells.
In this case, the appropriate restriction appears to be ${\rm sym}\, F\in Q^e_2\mathcal M$. 
Corollary \ref{cor:lowen}  asserts that $\{V^{qce}=0\}\subseteq Q^e_2\mathcal M$,
so that the restriction is mild; it would be interesting to know if further
restrictions could be imposed on $F$.

We remark that for a general $V$ it is not known whether
$\{ V^{qce}=0\} = Q^e\{V=0\}$, but it is always true (and it is easy to check)
that $ Q^e\{V=0\}\subseteq \{ V^{qce}=0\}$. 
On the other hand, in \cite[Theorem~4]{Zhang2004} it is proved that
\begin{equation*}\label{agosto}
{\bf Q}^e_q\tilde K=Q^e_q\tilde K=Q^e_1\tilde K, 
\quad\mbox{for every }q\in[1,\infty),
\end{equation*}
for every compact  $\tilde K\in\Rds$ but 
it is not known whether $Q_q^e\tilde K=Q^e\tilde K$ is true in general. 
The nonlinear version of these results is given by 
\cite[Proposition~2.5]{Zhang1998} and we have that
\begin{equation*}
QK={\bf Q}_qK=Q_q K,\quad\mbox{for every }q\in[1,\infty),   
\end{equation*}
for every compact $K\in\Rd$.

\begin{remark}
For energies describing nematic elastomers (materials to which our results apply, see Section \ref{sec:application}), 
more can be said for the geometrically linear as well as for the nonlinear case. Indeed, in the geometrically linear case,
the fact that $\{ V^{qce}=0\} = Q^e\{V=0\}$ is proved in \cite{Cesana} (see also \cite{Ce_DeSim}),  
while the nonlinear case $\{ W^{qc}=0\} = Q\{W=0\}$ is due to \cite{DeSim_Dolz}. 
\end{remark}

\begin{remark}[Inhomogeneous materials]
We conclude this section by noting that all the results stated hold in the more general case 
of inhomogeneous materials, that is when the energy densities $W_{\ep}$ are also functions
of $x\in\Om$. In this case, our hypotheses can be reformulated in the following way:
$W_{\ep}:\Om\times\Rd\to[0,\infty]$ is measurable, $W_{\ep}(x,\cdot)$ is continuous in an $(\ep,x)$-independent 
neighborhood of the identity and frame-indifferent, and $W_{\ep}(x,F)=0$ if and only if $F\in\mathcal U_{\ep}$
for a.e. $x\in\Om$. Moreover,
$$
W_\ep(x,F)\geq c g_p(d(F,\mathcal{U}_\ep))
$$
for all $F\in\mathbb{R}^{d\times d}$, for a.e. $x\in\Om$, and some $c>0$ independent of $\ep$ and $x$.
The functions $f_{\ep}(F)$ are replaced by $f_{\ep}(x,F)$ and we require that
$f_{\ep}\to f$ uniformly on $\Om\times K$ for every compact $K\in\Rd$, and that
$f(x,F)\leq c(1+|F|^2)$ for every $F\in\Rd$ and some constant $c>0$ independent of $x$.
Finally, $f^{qc}(F)$ has to be replaced
 by $f^{qc}(x,F)$, where
$f^{qc}(x,\cdot)=\left(f(x,\cdot)\right)^{qc}$ for a.e. $x\in\Om$.  
\end{remark}


\section{Application to Nematic Elastomers}\label{sec:application}

In this section, we consider the case $d=3$, so all deformations are maps from
$\RR^3$ to $\RR^3$. We use the notation $\tr F^2$ to denote the trace of the square
of a matrix $F\in\RR^{3\times3}$, while $\tr^2 F$ stands for $(\tr F)^2$.
The unit sphere of $\RR^3$ is denoted by $\Sph$.
 
We begin by recalling that the standard \emph{neo-Hookean} 
energy for incompressible deformations
\begin{equation}
  \label{neohook}
  W(F) = C\left(|F|^2-3\right),\quad {\rm if}\,\, \det F=1,
\end{equation}
with $C>0$, has a natural generalization to compressible strains (see \cite{Holz})
\begin{align}
  \label{neohookcomp0}
  W_{\rm comp}(F)  &= W((\det F)^{-1/3}F) + \Wvol (\det F)\\
  \label{neohookcomp}
  & = C\left(\frac{|F|^2}{(\det F)^{2/3}}-3\right) + 
\Wvol (\det F), \quad \det F>0.
\end{align}
The $1/3$ power in \eqref{neohookcomp0}
is natural because $\det[(\det F)^{-1/3}F]=1$, whenever $\det F>0$. 
We assume the function $\Wvol$ satisfies the following natural properties:
\begin{equation}
  \label{Wvol}\left.
  \begin{array}{l}
    \Wvol \in C^2((0,\infty),\RR),\\
    \Wvol (t) =0\,\mbox{ if and only if }\,t=1,\\
    \Wvol (t) \to +\infty, \,\, {\rm as} \,\, t\to 0^+,\\
    \Wvol(t) \geq k\,t^2 ,\,\, \mbox{for every }t\geq M >0,\,\mbox{for some }M,k>0,\\
    \Wvol''(1) >0. \\
  \end{array}\right\}
\end{equation}
In the condition $\Wvol''(1) >0 $, the strict inequality is
important for our analysis to apply, as will be apparent later. An example of $\Wvol$
is $\Wvol(t) = t^2-1-2\log t$.

We note that if $W_{\rm comp}$ is given by \eqref{neohookcomp} and $\Wvol$ satisfies
\eqref{Wvol}, then $W_{\rm comp}(F) \geq 0$ and $W_{\rm comp}(F) = 0$ if and only if $F\in SO(3)$.  
This can be seen by using a standard inequality between arithmetic
and geometric mean.

The transition from incompressible to compressible energies 
is the same for models of nematic elastomers. 
We begin by considering the standard energy density 
for modeling incompressible nematic elastomers given by
\begin{equation*}
  \label{Wincomp}
  W(F):=\min_{n\in \Sph}W_n (F),\quad {\rm if}\,\, \det F=1,
\end{equation*}
where
\begin{align*}
& W_n(F) := \frac{\mu}{2}\left[\tr(F^TL_n^{-1}F)-3\right],\\
& L_n := a^{2/3}n\otimes n + a^{-1/3}(I-n\otimes n), \label{Ln}
\end{align*}
and $\mu>0$,  $a>1$ are constants. 
This density has been studied, e.g., in \cite{AgDe,AgDe2,DeSim_Dolz,DeTe}.
Note that $n$ is an eigenvector
of $L_n$, with eigenvalue $a^{2/3}$. Any nonzero vector perpendicular to
$n$ is also an eigenvector of $L_n$, with eigenvalue $a^{-1/3}$. Hence,
$\det L_n = 1$. It is straightforward to check that
\begin{equation*}
L_n^{\alpha} = a^{2\alpha/3}n\otimes n + a^{-\alpha/3}(I-n\otimes n),\quad\mbox{for every }\alpha\in\RR.
\end{equation*}
Note that $W_n$ can be written in the neo-Hookean form
\[
W_n(F) = \frac{\mu}{2}\left[\tr(F_n^TF_n)-3\right],\quad F_n = L_n^{-\frac12}F,
\]
which shows that only the quantity $F_n$ related to the deformation gradient $F$ 
is responsible for the storage of energy.
Generalizing this form of $W_n$, just as $W$ in \eqref{neohook}
was replaced by $W_{{\rm comp}}$ in \eqref{neohookcomp},
we replace $W_n$ by
\begin{equation}
  \label{Wncomp}
  W_n(F)  := \frac{\mu}{2}\left[\frac{\tr(F^TL_n^{-1}F)}{\left(
\det F\right)^{2/3}}-3\right]+\Wvol (\det F),\quad \det F>0,
\end{equation}
where we have used the fact that $\det F_n=\det F$. 
We always assume that $\Wvol$ satisfies \eqref{Wvol}.

We work with the compressible model for nematic elastomers
given by the minimum over $n$ of the compressible densities $W_n$ in \eqref{Wncomp}:
\begin{equation*}
  \label{Wcomp}
W(F)  := \min_{n\in\Sph}\frac{\mu}{2}\left[\frac{\tr(F^TL_n^{-1}F)}{\left(
\det F\right)^{2/3}}-3\right]+\Wvol (\det F),\quad \det F>0.
\end{equation*}
A straightforward computation (cf. \cite{DeTe}) gives
\begin{equation}
  \label{min_form}
\min_{n\in \Sph}  \tr(F^TL_n^{-1}F)= 
\left(\frac{\lambda_1(F)}{a^{-1/6}}\right)^2 + 
\left(\frac{\lambda_2(F)}{a^{-1/6}}\right)^2 + 
\left(\frac{\lambda_3(F)}{a^{1/3}}\right)^2,
\end{equation}
where $0<\lambda_1(F)\leq \lambda_2(F) \leq \lambda_3(F)$ are the ordered singular values of $F$.
It is easy to check that $W(F) \geq 0$ and
$W(F)=0$ if and only if $F$ belongs to the set
\begin{equation*}
  \label{Unem}
  \mathcal U : = \left\{L_n^{\frac12}R \, :\, n \in \Sph, \, R\in SO(3)\right\}.
\end{equation*}
\begin{remark}\label{rem:wells}
In the sequel, we may equivalently consider wells of the form
\begin{equation*}
  \label{Unem2}
 \tilde{ \mathcal U} : = \left\{RL_n^{\frac12} \, :\, n \in \Sph, \, R \in SO(3)
\right\}.
\end{equation*}
Indeed, we have that $\mathcal U = \tilde{\mathcal U}$, because
$RL_n^{\frac12} = L_{Rn}^{\frac12}R$
for any $n\in \Sph$ and any $R\in SO(3)$.
\end{remark}


\subsection{The small-strain regime and its rigorous justification}
We consider the small-strain regime $a=(1+\ep)^3$, with $\ep \ll 1$.
In this case, we write
\begin{equation}
  \label{Lnep}
  L_{n,\ep}: = (1+\ep)^2n\otimes n + (1+\ep)^{-1}(I-n \otimes n),
\end{equation}
and using \eqref{min_form},
\begin{align}\label{Wep}
  W_\ep(F):&= \min_{n\in\Sph}\frac{\mu}{2}\left[\frac{\tr(F^TL_{n,\ep}^{-1}F)}{\left(
\det F\right)^{2/3}}-3\right]+\Wvol (\det F) \\
&=  \frac{\mu}{2}\left\{\frac{1}{(\det F)^{2/3}} \right.
\left[\left(\frac{\lambda_1(F)}{(1+\ep)^{-\frac12}}\right)^2 + 
\left(\frac{\lambda_2(F)}{(1+\ep)^{-\frac12}}\right)^2 \right.\nonumber\\
& \qquad \qquad \qquad \qquad \left. \left.+ 
\left(\frac{\lambda_3(F)}{1+\ep}\right)^2-3\right]\right\}+\Wvol (\det F)\label{Wep_2}
\end{align}
for all $F$ with $\det F>0$.
The set of wells for $W_\ep$ is
\begin{equation}
  \label{Uep}
 \mathcal{ U}_\ep : = \left\{RL_{n,\ep}^{\frac12} \, :\, n \in \Sph, \, R \in SO(3)
\right\}.
\end{equation}

\begin{remark}\label{rem:Un}
From \eqref{Lnep}, we can write $L_{n,\ep} = I + 2\ep U_n + o(\ep)$, where
\begin{equation}
  \label{Un}
  U_n := \frac12 \left( 3n\otimes n -I\right)
\end{equation}
is traceless. 
This definition of $U_n$ will be useful later on because we
will use the equivalence $(I+\ep E)^2-L_{n,\ep} = 2\ep(E-U_n) + o(\ep)$
to deduce the expression of the limiting small-strain energy density.
We note that using $a=(1+\ep)^\alpha$,
$\alpha\in\RR$, would also be a valid small-strain regime.
In this case the first-order expansion of $L_{n,\ep}$ would be the same as before but with
$\frac{\alpha}6(3n\otimes n-I)$ in place of (\ref{Un}).
The power $\alpha=3$ is chosen only for notational convenience. 
\end{remark}

The results of Section \ref{sec:main} can be applied with $p=3/2$ (see Lemma \ref{lem:Wgrowth}) 
to the model we have presented so far and 
in particular we can deduce Theorem \ref{thm:nematic} below. 
Here, $\Om$ is a bounded and Lipschitz domain of $\RR^3$ and,
as in Section \ref{sec:main}, $h\in W^{1,\infty}(\Om,\RR^3)$ represents the boundary data
and $\mathscr L:W^{1,\frac32}(\Om,\RR^3)\to\RR$, continuous and linear, represents the applied loads.
The set $W^{1,\frac32}_h$ is defined as in (\ref{Wph}) with $p=3/2$.

\begin{theorem}\label{thm:nematic}
Consider the family of energy densities given by (\ref{Wep}).
Set 
\[
m_{\ep}:=\inf_{u\in W^{1,\frac32}_h}\left\{\frac{1}{\ep ^2}\int_{\Omega}W_\ep(I+\ep\nabla u)dx-\mathscr L(u)\right\},
\]
and suppose that $\{u_{\ep}\}$ is a sequence such that
\begin{equation*}
\lim_{\ep\to0}\left\{\frac{1}{\ep ^2}\int_{\Omega}W_\ep(I+\ep\nabla u_\ep)dx-\mathscr L(u_{\ep})\right\}
=\lim_{\ep\to0}m_{\ep}.
\end{equation*}
Then, up to a subsequence, $u_{\ep}\rightharpoonup u$ in $W^{1,\frac32}(\Om,\RR^3)$, 
where $u$ is a solution to the minimum problem
\begin{equation*}
m:=\min_{u\in W_h^{1,2}}\left\{\int_{\Om}f^{qc}(\na u)dx-\mathscr L(u)\right\},
\end{equation*}
with
\begin{equation}\label{eqn:exp_f}
f(F):=\mu \min_{n\in\Sph}|{\rm sym}\,F-U_n|^2 +\frac{\lambda}{2}\tr^2F,
\qquad\lambda:=\Wvol ''(1)-\frac23\mu,
\end{equation}
for every $F\in\RR^3$. Moreover, $m_{\ep}\to m$.
\end{theorem}

The choice of notation $\mu$ and $\lambda$ for the constants in (\ref{eqn:exp_f}) 
is motivated by the theory of isotropic linear
elasticity where $\mu$ and $\lambda$ correspond to the shear and bulk modulus, respectively.

We remark that an explicit expression for the quasiconvexification $f^{qc}$ of $f$, 
as given in \eqref{fqc_form} below,
is due to \cite{Cesana}. It turns out that, while the expression of $f$ involves
the distance of ${\rm sym}\,F$ to the set of matrices 
$$
\{U_n:n\in\Sph\}=\Big\{U\in\RR^{3\times 3}_{{\rm sym}}:\{\mbox{eigenvalues of }U\}=\{-1/2,-1/2,1\}\Big\}
$$ 
(where $U_n$ is defined in (\ref{Un})), 
the expression of $f^{qc}$ involves the distance of ${\rm sym}\,F$ to the set
$$
\mathscr Q:=\Big\{U\in\RR^{3\times 3}_{{\rm sym}}:\tr\,U=0\mbox{ and }\{\mbox{eigenvalues of }U\}\subseteq[-1/2,1]\Big\}.
$$
More precisely,
\begin{equation}
  \label{fqc_form}
f^{qc}(F)=\mu\min_{U\in\mathscr Q}|{\rm sym}\,F-U|^2 +\frac{\lambda}2\tr^2F,  
\end{equation}
for every $F\in\RR^{3\times 3}$. 

Theorem \ref{thm:nematic} is a straightforward application of Theorem \ref{main_thm}
once we establish that the family of energies $\{W_{\ep}\}$ given by (\ref{Wep})
satisfies the hypotheses. Essentially, we have to verify that the growth condition (\ref{eq:coercivity})
with $p=3/2$ and $\mathcal U_\ep$ given by (\ref{Uep}) is satisfied (see Lemma \ref{lem:Wgrowth}). 
Also, we have to show that $f$ as in (\ref{eqn:exp_f}) satisfies
$f(F)=V({\rm sym}\,F)$, where $W_{\ep}(I+\ep\cdot)/\ep^2\to V$, as $\ep\to 0$, 
uniformly on compact subsets of $\RR^{3\times 3}_{{\rm sym}}$ 
(see Lemma \ref{prop:lin_lim}).
The fact that $0\leq f(F)\leq c(1+\vert F\vert^2)$ is a direct consequence of (\ref{eqn:exp_f}).

In establishing estimates, it is useful to define the functions
\begin{equation}
  \label{tildeWB}
  \widetilde W_{n,\ep}(B): =\frac{\mu}{2}\left(\frac{\tr(BL_{n,\ep}^{-1})}{(\det B)^{1/3}}-3\right) + 
\Wvol (\sqrt{\det B\,}),
\quad  \widetilde W_{\ep}(B)  := \min_{n\in\Sph}\widetilde W_{n,\ep}(B),
\end{equation}
for every positive definite, symmetric matrix $B$,
to replace $W_{n,\ep}$ and $W_\ep$.
Indeed, 
\begin{equation}
  \label{tildeW}
  W_{n,\ep}(F)=\widetilde W_{n,\ep}(FF^T),\qquad W_{\ep}(F)=\widetilde W_{\ep}(FF^T),
\end{equation}
for every $F\in\RR^{3\times 3}$ with $\det F>0$.

\begin{lemma}
  \label{lem:Wgrowth}
Let $W_\ep$ be defined as in \eqref{Wep}, and $\mathcal U_\ep$
as in \eqref{Uep}. Then the following holds:
\begin{itemize}
\item[(i)] $W_\ep (F) \geq c\,{\rm dist}^2(F,\mathcal U_\ep)$ for $F$ near $\mathcal U_\ep$,
\item[(ii)] $W_\ep (F) \geq c\,{\rm dist}^{3/2}(F,\mathcal U_\ep)$ 
for $F$ far from $\mathcal U_\ep$.
\end{itemize}
\end{lemma}
Note that the above lemma implies that, up to a multiplicative constant, $W_\ep(F)$ 
is bounded below by $g_{\frac32}(\dist(F,\mathcal U_\ep))$, where $g_{\frac32}$ is the function given 
by (\ref{eq:gp}) with $p=3/2$.
\begin{proof}
Note that $\widetilde W_{n,\ep}(B)$ is minimized at the level $0$ 
by $B=L_{n,\ep}$, so that Taylor expansion gives
\begin{equation}
\label{eq:taylor_tilde}
\widetilde W_{n,\ep}(B) = \frac12 D^2 \widetilde W_{n,\ep}(L_{n,\ep})[B-L_{n,\ep}]^2 + 
o\left(|B-L_{n,\ep}|^2\right). 
\end{equation}
A direct computation yields

\begin{equation}
  \label{D2WH}
  D^2 \widetilde W_{n,\ep}(L_{n,\ep})[H]^2 
= \left(\frac14\Wvol ''(1)-\frac{\mu}{6}\right)
\tr^2(H L_{n,\ep}^{-1})+\frac{\mu}{2}\tr{(H L_{n,\ep}^{-1})^2},
\end{equation}
for every $H\in\Rds$.
Hence, if $\frac14\Wvol ''(1)-\frac{\mu}{6}\geq 0$ then 
$D^2 \widetilde W_{n,\ep}(L_{n,\ep})[H]^2\geq \frac{\mu}{2}\tr{(H L_{n,\ep}^{-1})^2}$.
On the other hand, if $\frac14\Wvol ''(1)-\frac{\mu}{6} < 0$ then
\begin{equation*}
  D^2 \widetilde W_{n,\ep}(L_{n,\ep})[H]^2 
\geq \frac34 \Wvol ''(1) \tr{ (H L_{n,\ep}^{-1})^2}.
\end{equation*}
This is due to the fact that 
$\tr(H L_{n,\ep}^{-1}) = \tr(L_{n,\ep}^{-\frac12} H L_{n,\ep}^{-\frac12})$
and that 
$(\tr{A})^2\leq 3\,\tr{A^2}$ for any $A\in \RR^{3\times 3}_{{\rm sym}}$. Thus,

\begin{equation}
  \label{DWne1}
  D^2 \widetilde W_{n,\ep}(L_{n,\ep})[H]^2\geq \min\left\{ \frac{\mu}{2},\frac34\Wvol''(1) \right\} \tr{ (H L_{n,\ep}^{-1})^2}.
\end{equation}
Now, since $\Sph$ is compact, one can show that
\begin{equation}
  \label{trace}
   \tr{ (H L_{n,\ep}^{-1})^2} \geq 
   \frac12 |H|^2,
\end{equation}
for every $n\in \Sph$, $H\in \Rds$, and for all $\ep$ small enough (independently of $n$ and $H$).
Then, from (\ref{eq:taylor_tilde}), (\ref{DWne1}) and (\ref{trace}), we obtain
\begin{equation}
  \label{DWne2}
  \widetilde W_{n,\ep}(B) 
\geq \frac14\min\left\{ \frac{\mu}{2},\frac34\Wvol''(1) \right\} |B-L_{n,\ep}|^2 + 
o\left(|B-L_{n,\ep}|^2\right). \nonumber
\end{equation}
Thus, provided $|B-L_{n,\ep}|$ is small, we have that
\begin{align}
  W_{n,\ep}(F) &= \widetilde W_{n,\ep}(FF^T) \geq c\left|FF^T-L_{n,\ep}\right|^2 
\geq c\left| \sqrt{FF^T}-L_{n,\ep}^{\frac12} \right|^2\nonumber\\
&\geq c\min_{Q\in SO(3)} \left|F -L_{n,\ep}^{\frac12}Q \right|^2
 \geq  c\min_{n\in\Sph}\min_{Q\in SO(3)} \left|F -L_{n,\ep}^{\frac12}Q \right|^2\nonumber\\
 & = c\,{\rm dist}^2(F,\mathcal U_\ep),
 \label{eq:near_estimate}
\end{align}
where we have also used the fact that $|\sqrt{F}-\sqrt{G}| \leq c |F-G|$ 
for any two positive definite matrices $F$ and $G$ sufficiently close to the identity.
Since $c$ in (\ref{eq:near_estimate}) is independent of $n$, we then have 
$\min_{n\in\Sph}W_{n,\ep}(F)\geq c\,{\rm dist}^2(F,\mathcal U_\ep)$,
whenever ${\rm dist}(F,\mathcal U_\ep)$ and $\ep$ are sufficiently small. This establishes (i).

Without loss of generality, we can assume $\ep \leq 1$ so that $1\leq (1+\ep)^2\leq 4$. Hence,
\[
\left(\frac{\lambda_1}{(1+\ep)^{-\frac12}}\right)^2+\left(\frac{\lambda_2}{(1+\ep)^{-\frac12}}\right)^2
+\left(\frac{\lambda_3}{1+\ep}\right)^2
\geq \lambda_1^2 + \lambda_2^2 +  \frac{\lambda_3^2}{4},
\]
and from \eqref{Wep_2}
\begin{equation}\label{bound1}
  W_\ep(F) 
\geq \frac{\mu}{2}\left(\frac{|F|^2}{4(\det F)^{2/3}}-3\right) +\Wvol(\det F).
\end{equation}
There are two cases. Either $\det F <M$ or $\det F \geq M$, where $M$ is the 
constant in \eqref{Wvol}. In the case  $\det F <M$, from \eqref{bound1} we obtain
\begin{align*}
  W_\ep(F)\geq c_1|F|^2-c_2 
  \geq  c\,{\rm dist}^2(F,\mathcal U_\ep)
\end{align*}
for $|F|\gg 1$, since $\Wvol \geq 0$.
Now, if $\det F \geq M$, we know from \eqref{Wvol} that $\Wvol (\det F)\geq k(\det F)^2$. Hence,
it follows from \eqref{bound1} that
\[
W_\ep(F) \geq \min \left\{ \frac{\mu}{8}, k\right\}\left(\frac{|F|^2}{(\det F)^{2/3}}+
(\det F)^2\right) -\frac{3\mu}{2}.
\]
Applying Young's inequality $xy \leq \frac1r x^r +\frac1q y^q$
with $x = (\det F)^{-1/2}|F|^{3/2}$, $y = (\det F)^{1/2}$, and $r=4/3$, $q=4$, we  have
\begin{equation*}
  W_\ep(F)
  \geq \frac43  \min \left\{ \frac{\mu}{8}, k\right\} \left(\frac34 x^{4/3}+\frac14 y^4\right) -\frac{3\mu}{2}\\
\geq\frac43  \min \left\{ \frac{\mu}{8}, k\right\}|F|^{3/2} -\frac{3\mu}{2}.
\end{equation*}
Thus, for $|F|\gg 1$, 
$W_\ep(F)
\geq  c\,{\rm dist}^{3/2}(F,\mathcal U_\ep)$.
\end{proof}

\begin{lemma}\label{prop:lin_lim}
Let $U_n$ be defined as in \eqref{Un}. For $E\in \RR^{3\times 3}_{{\rm sym}}$, we define
\begin{equation*}
  \label{Vform}
  V(E) = \mu \min_{n\in\Sph}|E-U_n|^2 +\frac{\lambda}{2}\tr^2E,
\qquad \lambda = \Wvol ''(1)-\frac23\mu.
\end{equation*}
Then 
\begin{equation}
  \label{VWlim}
  V(E) = \lim_{\ep \to 0}\frac{1}{\ep^2}W_\ep(I+\ep E),
\end{equation}
where $W_\ep$ is defined in \eqref{Wep}, and the limit is uniform on compact subsets of $\RR^{3\times 3}_{{\rm sym}}$.
\end{lemma}


\begin{proof}
For every $E\in\RR^{3\times 3}_{{\rm sym}}$, let us define
\begin{equation*}
  \label{Vn}
  V_n(E):=2D^2\widetilde W_{n,0}(I)[E-U_n]^2,\qquad \tilde V(E):=\min_{n\in\Sph}V_n(E),
\end{equation*}
where $\widetilde W_{n,0}$ is given by (\ref{tildeWB}) with $\ep=0$.
Note that from \eqref{D2WH} with $\ep=0$ we have that
\[
D^2\widetilde W_{n,0}(I)[H]^2 = \left(\frac14\Wvol ''(1) -\frac{\mu}{6}\right)\tr^2H +\frac{\mu}{2}|H|^2,
\]
for every $H\in\RR^{3\times 3}_{{\rm sym}}$, so that 
\begin{align*}
\tilde V(E) & =\min_{n\in\Sph} \left(\frac12 \Wvol ''(1) 
                     -\frac{\mu}{3}\right)\tr^2(E-U_n)+\mu|E-U_n|^2 \\
 & =  \min_{n\in\Sph} \mu|E-U_n|^2 +\frac{\lambda}{2} \tr^2E,
\end{align*}
with $\lambda = \Wvol ''(1)-\frac23\mu$, in view of the fact that $\tr\,U_n=0$.
Thus, $\tilde V=V$ where $V$ is given by (\ref{VWlim}). 
Now, for $W_{n,\ep}$ defined in \eqref{Wncomp} with
$L_{n,\ep}$ in place of $L_n$, let us introduce for every $E\in \RR^{3\times 3}_{{\rm sym}}$
\begin{equation*}
  \label{qdef}
  q_{n,\ep}(E) : = \frac{1}{\ep^2}W_{n,\ep}(I+\ep E),\quad 
q_\ep (E) : = \min_{n\in \Sph}q_{n,\ep}(E)=\frac1{\ep^2}W_{\ep}(I+\ep E).
\end{equation*}
To prove the lemma,
we show that $q_{\ep_j}\to\tilde V$, uniformly on compact subsets of $\Rds$,
for every vanishing sequence $\{\ep_j\}$.

Given a compact $K\subset \RR^{3\times 3}_{{\rm sym}}$, 
we  prove that $\sup_K (q_{\ep_j} - \tilde V) \to 0$ and $\inf_K (q_{\ep_j} - \tilde V) \to 0$, so that
$\sup_K |q_\ep - \tilde V|=\max\{\sup_K (q_{\ep_j} - \tilde V),-\inf_K (q_{\ep_j} - \tilde V)\} \to 0$. Note
that $q_{\ep_j}$ and $\tilde V$ are both continuous so
\begin{equation*}
  \label{Eep}
  \sup_K (q_{\ep_j} - \tilde V) =
q_{\ep_j} (E_{\ep_j})-\tilde V(E_{\ep_j})
\end{equation*}
for some $E_{\ep_j}\in K$. Up to subsequences, we have that $E_{\ep_j}\to E$, as $j\to\infty$,
for some $E\in K$.  
For any $n\in \Sph$,
\begin{equation}
  \label{supEep}
  \sup_K(q_{\ep_j}-\tilde V) 
\leq q_{n,\ep_j}(E_{\ep_j})-\tilde V(E_{\ep_j}) 
\leq \sup_K |  q_{n,\ep_j} -V_n| + V_n(E_{\ep_j})-\tilde V(E_{\ep_j}) .
\end{equation}
By Lemma \ref{lem:limit1} below, $q_{n,\ep} \to V_n$ uniformly on $K$, as $\ep \to 0$, for every $n\in\Sph$.
Therefore, by continuity of $V_n$, we obtain from (\ref{supEep}) that
\[
\limsup_{j\to \infty} \, \sup_K(q_{\ep_j}-\tilde V)\leq V_n(E)-\tilde V(E),\quad\mbox{for every }n\in \Sph.
\]
Taking the minimum over $n\in \Sph$ implies 
\begin{equation}
  \label{limsup}
  \limsup_{j \to \infty} \, \sup_K(q_{\ep_j}-\tilde V)\leq \min_{n\in \Sph }V_n(E)-\tilde V(E)=0.
\end{equation}

On the other hand, 
\[
\sup_K (q_{\ep_j}-\tilde V) = q_{\ep_j}(E_{\ep_j})-\tilde V(E_{\ep_j}) = 
q_{n_j,\ep_j}(E_{\ep_j})-\tilde V(E_{\ep_j}) ,
\]
where $n_j\in \Sph$ attains the minimum. Up to taking a further subsequence,
we may assume that $n_j\to\hat n$ as $j\to \infty$. Now
\begin{equation*}
|q_{n_j,\ep_j}(E_{\ep_j})-V_{\hat n}(E)|
\leq \sup_K |q_{n_j,\ep_j} - V_{\hat n} | + | V_{\hat n}(E_{\ep_j})-V_{\hat n}(E)| ,
\end{equation*}
and 
$\sup_K |q_{n_j,\ep_j}  - V_{\hat n} |\to 0$ by elementary computations.
Thus, $q_{n_j,\ep_j}(E_{\ep_j})\to V_{\hat n}(E)$ and 
\begin{equation}
  \label{liminf}
  \liminf_{j\to \infty} \, \sup_K(q_{\ep_j}-\tilde V) = V_{\hat n}(E) - \tilde V(E)
\geq \min_{n\in \Sph} V_n (E) - \tilde V(E) = 0.
\end{equation}
Together, \eqref{limsup} and \eqref{liminf} imply
\begin{equation}
  \label{suplimit}
\lim_{j\to\infty}\sup_K(q_{\ep_j}-\tilde V)=0.
\end{equation}

Establishing \eqref{suplimit} with $\sup_K$ replaced by $\inf_K$ is very similar. 
Indeed, let $E_{\ep_j}$ be such that 
\begin{equation*}
  \label{Eep2}
  \inf_K (q_{\ep_j} - \tilde V)=q_{\ep_j}(E_{\ep_j})-\tilde V(E_{\ep_j}),
\end{equation*}
to obtain an $E\in K$ and $\{n_j\}\subset\Sph$ such that
$E_{\ep_j}\to E$, $n_j\to\hat n$, and $q_{\ep_j}(E_{\ep_j})=q_{n_j,\ep_j}(E_{\ep_j})$.
For any $n\in \Sph$,
$\inf_K(q_{\ep_j}-\tilde V) \leq q_{n,\ep_j}(E_{\ep_j})-\tilde V(E_{\ep_j})$ and,
just as  before, $q_{n,\ep_j}(E_{\ep_j})\to V_n(E)$. Thus 
\[
\limsup_{j\to \infty} \, \inf_K(q_{\ep_j}-\tilde V) \leq V_n(E)-\tilde V(E),\quad\mbox{for every }n\in \Sph ,
\]
and in turn
\begin{equation}\label{limsup2}
\limsup_{j\to \infty} \, \inf_K (q_{\ep_j}-\tilde V)\leq\min_{n\in\Sph}V_n(E)-\tilde V(E)=0.
\end{equation}
On the other hand, as before $q_{n_j,\ep_j}(E_{\ep_j})\to V_{\hat n}(E)$ so that
\begin{align*}
   \liminf_{j\to \infty} \, \inf_K(q_{\ep_j}-\tilde V) &= \liminf_{j\to \infty}
\left(q_{n_j,\ep_j}(E_{\ep_j})-\tilde V (E_{\ep_j})\right)\nonumber\\
& = V_{\hat n}(E) -\tilde V(E) \geq \min_{n\in\Sph}V_{\hat n}(E) -\tilde V(E) =0.
\end{align*}
This, together with (\ref{limsup2}), implies
\begin{equation}
  \label{inflimit}
  \lim_{j\to \infty} \inf_{K}(q_{\ep_j}-V) = 0.
\end{equation}
Equations \eqref{suplimit} and \eqref{inflimit} complete the proof.

\end{proof}

\begin{lemma}\label{lem:limit1}
  For all $n\in \Sph$, $q_{n,\ep} \longrightarrow V_n$, as $\ep \to 0$,  uniformly on compact subsets of
$\RR^{3\times 3}_{{\rm sym}}$. 
\end{lemma}

\begin{proof}
Let $K\subset \RR^{3\times 3}_{{\rm sym}}$ be compact. Recall from \eqref{tildeW} 
that for every $E\in\RR^{3\times 3}_{{\rm sym}}$,
$W_{n,\ep}(I+\ep E) = \widetilde W_{n,\ep}((I+\ep E)^2)$ and that
$(I+\ep E)^2-L_{n,\ep} = 2\ep (E-U_n)+o(\ep)$, from Remark \ref{rem:Un}.
Thus, for every $E\in K$ we have by Taylor expansion that
  \begin{align*}
q_{n,\ep}(E)
& =\frac1{2\ep^2}D^2\widetilde W_{n,\ep}(L_{n,\ep})[(I+\ep E)^2-L_{n,\ep}]^2
+\frac1{\ep^2}o(|(I+\ep E)^2-L_{n,\ep}|^2)\\
& = 2D^2\widetilde W_{n,\ep}(L_{n,\ep})[E-U_n+o(1)]^2  + o(1).
\end{align*}
Adding and subtracting $2D^2\widetilde W_{n,0}(L_{n,\ep}))[  E-U_n+o(1)]^2$ and taking the
supremum over $E\in K$ gives
\begin{align*} 
\sup_K|q_{n,\ep}-V_n| 
&\leq 2\sup_{E\in K} \left\{\left|(D^2\widetilde W_{n,\ep}(L_{n,\ep})
-D^2\widetilde W_{n,0}(L_{n,\ep}))[  E-U_n+o(1)]^2\right|\right. \\
&+\left. \left| D^2\widetilde W_{n,0}(L_{n,\ep})[  E-U_n+o(1)]^2-\frac{V_n(E)}2+o(1)\right|\right\} \\
&\leq C\sup_{M\in K} \left|D^2\widetilde W_{n,\ep}(M)-D^2\widetilde W_{n,0}(M)\right|\\
&\qquad \qquad \qquad +C \left|D^2\widetilde W_{n,0}(L_{n,\ep})-D^2\widetilde W_{n,0}(I)\right|+o(1),
  \end{align*}
where in the last inequality we have used the definition of $V_n$.
By elementary computations one can verify that the 
summands on the right side of the last inequality
tend to $0$ as $\ep\to0$.
\end{proof}


\section{Proofs of the main results}\label{sec:cpt}

For the sequential characterization of $\Gamma$-convergence, as 
well as to prove that almost minimizers of $\mathscr E_\ep$ 
converge to minimizers of $\overline{\mathscr E}$, we need 
to establish that the functionals $\mathscr E_\ep$ are equicoercive; 
this is the content of Proposition~\ref{prop:compactness}. 
Before proving it, we collect some useful properties of the function 
$g_p$ defined in \eqref{eq:gp}.

\begin{lemma}
\label{lemma:gp}
The function $g_p$ satisfies the following:
\begin{itemize}
\item[(i)] $g_p$ is convex;
\item[(ii)] $g_p(s+t)\leq C(g_p(s)+t^2)$ for all $s,t\geq0$, where
 $C>0$ is a constant depending only on $p$;
\item[(iii)] for each $K>0$ there exists a constant $C>0$ depending on $K$ and $p$ such that
\begin{eqnarray*}
t^2\leq C g_p(t),&& 0\leq t\leq K\\
t^p\leq C g_p(t),&& t\geq K.
\end{eqnarray*}

\end{itemize}
\end{lemma}

The proof of this lemma is elementary and left to the reader.



\begin{proof}[Proof of Proposition \ref{prop:compactness}]
We may assume that $u\in W^{1,p}_{h}$, otherwise the result follows trivially. 
Supposing that the bound
\begin{equation}
\label{eq:compactness0}
\int_{\Omega}g_p(\vert\ep\nabla u(x)\vert)dx\leq C\ep^2(1+\mathscr E_\ep(u))
\end{equation}
holds,  Proposition \ref{prop:compactness} can be established 
by estimating $\Vert\ep\nabla u\Vert^{p}_{p}$. Indeed,
by H\"older's inequality and the definition of $g_p$ in (\ref{eq:gp}),
\[
\int_{\left\{x\in\Omega:\vert\ep\nabla u(x)\vert\leq1\right\}}\vert\ep\nabla u(x)\vert^pdx 
\leq C\left(\int_{\left\{x\in\Omega:\vert\ep\nabla u(x)\vert\leq1\right\}}g_p(\vert\ep\nabla u(x)\vert)dx\right)^{\frac{p}{2}}.
\]
Thus, using \eqref{eq:compactness0} and the fact that
$t^{p/2}\leq t+1$ for every $t\geq0$, we obtain 
\begin{equation}
\label{eq:compactness1}
  \int_{\left\{x\in\Omega:\vert\ep\nabla u(x)\vert\leq1\right\}}\vert\ep\nabla u(x)\vert^pdx \leq
C\ep^p\left(1+\mathscr E_\ep(u)\right)^{\frac{p}{2}} 
\leq  2C\ep^p (1+\mathscr E_\ep(u)),
\end{equation}
because $\mathscr E_\ep(u)\geq0$.
On the other hand, by Lemma~\ref{lemma:gp} (iii) and  \eqref{eq:compactness0},
\begin{eqnarray}
\int_{\left\{x\in\Omega:\vert\ep\nabla u(x)\vert >1\right\}}\vert\ep\nabla u(x)\vert ^pdx&\leq & C \int_{\left\{x\in\Omega:\vert\ep\nabla u(x)\vert >1\right\}}g_p(\vert\ep\nabla u(x)\vert)\nonumber\\
&\leq & C\ep ^2(1+\mathscr E_\ep(u))
\leq C\ep ^p(1+\mathscr E_\ep(u)),
\label{eq:compactness2}
\end{eqnarray}
since $p\leq 2$ and $\ep$ is small. The compactness result now follows by (\ref{eq:compactness1}) and (\ref{eq:compactness2}). Hence, to complete the proof, 
we need only establish (\ref{eq:compactness0}).

Note that, by the coercivity condition (\ref{eq:coercivity}),
\[
\int_{\Omega}g_p(\dist (I+\ep\nabla u(x),\mathcal{U}_\ep))\,dx\leq C\ep ^2\mathscr E_\ep(u)
\]
and, since $\mathcal{U}_\ep$ is compact, there exist 
$R_\ep(x)U_\ep (x) \in\mathcal{U}_\ep$ 
such that the distance is achieved for a.e.~$x\in\Omega$, i.e.
\begin{equation}
\label{eq:compactness3}
\int_{\Omega}g_p(\vert I+\ep\nabla u(x)-R_\ep(x)U_\ep(x)\vert)\,dx
\leq C\ep ^2\mathscr E_\ep(u).
\end{equation}

In order to apply the rigidity result of Friesecke, James and M\"uller \cite{FrieJamesMuller} 
we need a lower bound for $g_p(\vert 1+\ep\nabla u(x)-R_\ep(x)U_\ep(x)\vert)$ 
in terms of the distance of $I+\ep\nabla u(x)$ to $SO(d)$. 
But since $g_p$ is increasing, by Lemma~\ref{lemma:gp} (ii) we infer that for $v(x)=x+\ep u(x)$,
\begin{eqnarray}\label{gp_est1}
g_p(\dist (\nabla v(x), SO(d)))&\leq & g_p(\vert \nabla v(x)-R_\ep(x)\vert)\nonumber \\
&\leq &g_p(\vert \nabla v(x)-R_\ep(x)U_\ep(x)\vert +\vert R_\ep(x)U_\ep(x)-R_\ep(x)\vert)\nonumber \\
&\leq & c\{g_p(\vert \nabla v(x)-R_\ep(x)U_\ep(x)\vert)+\vert U_\ep(x) -I\vert ^2\}
\end{eqnarray}
But $U_\ep(x)=I+\ep U(x)+o(\ep)$ for some $U(x)$ in the compact set $\mathcal{M}$ implies that
\begin{equation*}
\vert U_\ep(x) -I\vert =\ep\vert U(x)-o(1)\vert\leq c\,\ep
\end{equation*}
so that, by \eqref{eq:compactness3} and \eqref{gp_est1}, 
\begin{equation}
\label{eq:compactness4}
\int_\Omega g_p(\dist (\nabla v(x), SO(d)))\,dx
\leq c\,\ep^2(\mathscr E_\ep(u) +1).
\end{equation}
Now we may apply the modified rigidity result of \cite{FrieJamesMuller} 
(cf.~\cite[Lemma 3.1]{AgDalDe}) 
to get the existence of an $x$-independent $R_\ep\in SO(d)$ such that,
in conjunction with (\ref{eq:compactness4}),
\begin{equation}
\label{eq:compactness5}
\int_{\Omega}g_p(\vert\nabla v(x)-R_\ep\vert)\,dx
\leq c\int_{\Omega}g_p(\dist (\nabla v(x), SO(d)))\,dx
\leq c\,\ep^2(\mathscr E_\ep(u) +1).
\end{equation}
Also, by \cite[Lemma 3.3]{AgDalDe}, we infer that
\begin{equation}
\label{eq:compactness6}
\vert R_\ep -I\vert ^2
\leq c\,\ep^2 \left[\mathscr E_\ep(u) 
    + \left(\int_{\partial_D\Omega}|h|d\mathscr{H}^{n-1}\right)^2\right] 
    \leq c\,\ep^2(\mathscr E_\ep(u) +1),
\end{equation}
where $c$ now depends on $h$ and $\partial_D\Omega$ as well.
To complete the proof, note that as before, $g_p$ being increasing, 
a further application of Lemma~\ref{lemma:gp} (ii) shows that
\begin{eqnarray*}
\int_{\Omega}g_p(\vert\ep\nabla u(x)\vert)dx&\leq &\int_{\Omega}g_p(\vert I+\ep\nabla u(x)-R_\ep\vert +\vert R_\ep -I\vert )\\
&\leq &C\int_{\Omega}g_p(\vert I+\ep\nabla u(x)-R_\ep\vert) + \vert R_\ep -I\vert ^2dx\\
&\leq & C\ep ^2(I+\mathscr E_\ep(u))
\end{eqnarray*}
by (\ref{eq:compactness5}) and (\ref{eq:compactness6}), establishing (\ref{eq:compactness0}).
\end{proof}

Before proving Theorem \ref{thm:gammaconvergence},
we state two auxiliary key results which are used in the proof. 
For the proof of Lemma~\ref{lemma:bj}, 
we refer the reader to~\cite{AgDalDe}; Lemma~\ref{lemma:kristensen} 
is due to J.~Kristensen \cite{Kristensen}. 
In what follows, given a set $B\subset\mathbb{R}^d$, we denote by $1_{B}$ its characteristic function.

\begin{lemma}
\label{lemma:bj}
Let $\ep_j\rightarrow0$, as $j\rightarrow\infty$. 
Suppose that $\left\{\mathscr E_{\ep_j}(u_j)\right\}$ is bounded for some sequence 
$\{u_j\}\subset W^{1,p}(\Omega,\mathbb{R}^d)$ and that, in view of compactness, 
$u_j\rightharpoonup u$ in $W^{1,p}(\Omega,\mathbb{R}^d)$. For each $j$, define the sets
\begin{equation}
\label{eq:bj}
B_j:=\left\{x\in\Omega : \vert\nabla u_j(x)\vert\leq\frac{1}{\sqrt{\ep_{j}}}\right\}.
\end{equation}
Then, the following holds:
\begin{itemize}
\item[(i)] $|B^c_j|\to 0$ and $1_{B^c_j}\nabla u_j\to 0$ in $L^q(\Omega,\Rd)$, for all $1\leq q < p$;
\item[(ii)] $\nabla u\in L^2(\Omega,\mathbb{R}^{d\times d})$ and $1_{B_j}\nabla u_j\rightharpoonup\nabla u$ in $L^2(\Omega,\mathbb{R}^{d\times d})$.
\end{itemize}

\end{lemma}

\begin{lemma}[Proposition 1.10, \cite{Kristensen}]
\label{lemma:kristensen}
Let $g:\Rd\to\RR$ be a quasiconvex function such that for some $c_1,\,c_2>0$ and $p>1$
\[
c_1|F|^p - c_2 \leq g(F) \leq c_2(1 + |F|^p),\quad\mbox{for all $F\in\Rd$.}
\]
Then there exists a nondecreasing sequence of quasiconvex functions $\psi_k:\Rd\to\RR$, 
bounded above by $g$, such that $\{\psi_k\}$ converges to $g$ pointwise and
\begin{equation*}
\psi_k(F) = a_k|F| + b_k,\quad\mbox{for all $|F|\geq r_k$,}
\end{equation*}
for some $a_k,\,r_k>0$ and $b_k\in\RR$.
\end{lemma}

\begin{proof}[Proof of Theorem \ref{thm:gammaconvergence}]
It suffices to show that, fixing a vanishing sequence $\{\ep_j\}$, 
$\{\mathscr E_{\ep_j}\}$ $\Gamma$-converges to $\overline{\mathscr E}$.

To establish the $\Gamma$-$\liminf$ inequality, we follow the 
lines of the proof of the $\Gamma$-$\liminf$ inequality in \cite{Schmidt} with
some modifications.

Let $u_j\rightharpoonup u$ in $W^{1,p}(\Omega,\mathbb{R}^d)$ and assume 
that $\liminf\mathscr E_{\ep_j}(u_j)<\infty$, as otherwise, the result follows trivially. 
In particular, we may assume that, up to a subsequence, $\mathscr E_{\ep_j}(u_j)\leq c<\infty$,
so that, in particular, $u_j$, $u\in W^{1,p}_h$, and
\[
\mathscr E_{\ep_j}(u_j)\geq
\int_{B_j}\frac{1}{\ep^{2}_{j}}
W_{\ep_j}\left(\sqrt{(I+\ep_j\nabla u_j)^T(I+\ep_j\nabla u_j)}\right)\,dx,
\]
in view of frame-indifference. Note that the above integral 
is taken over the set $B_j$, defined in \eqref{eq:bj}, where we may assume that 
the determinant of $I+\ep_j\nabla u_j$ is bounded away from zero and apply the 
polar decomposition.
It is useful to introduce the function $\zeta:\Rd\to\Rds$ defined as
\begin{equation}\label{eq:zeta}
\zeta(F):=\sqrt{(I+F)^T(I+F)}-I-{\rm sym}\,F,
\end{equation}
which satisfies
\begin{equation}\label{eq:prop_h}
\zeta(F)\leq c\min\{|F|,|F|^2\},\qquad\mbox{for every }F\in\Rd,
\end{equation}
and to use the notation \eqref{lin_lim_1}-\eqref{eq:fep} to write
\begin{align*}
W_{\ep_j}\left(\sqrt{(I+\ep_j F)^T(I+\ep_j F)}\right)&=
       W_{\ep_j}\left(I+\ep_j\left({\rm sym}\,F+\frac{\zeta(\ep_j F)}{\ep_j}\right)\right)\\
   &=\ep_j^2 f_{\ep_j}\left(F+\frac{\zeta(\ep_j F)}{\ep_j}\right).
\end{align*}
Note also that we may exploit the boundedness of 
the sequence $\{1_{B_j}\na u_j\}$ in $L^2(\Om,\Rd)$ given by Lemma \ref{lemma:bj} 
to get that
\begin{equation}\label{eq:E_jf_j}
\mathscr E_{\ep_j}(u_j)\geq\int_{B_j}
      f_{\ep_j}\left(\na u_j+\frac{\zeta(\ep_j\na u_j)}{\ep_j}\right)\,dx.
\end{equation}
Consider the function $f$ which is the uniform limit of the sequence
$\{f_{\ep_j}\}$ on compact subsets of $\Rd$.
To employ the approximation result of Lemma~\ref{lemma:kristensen} which
requires quadratic growth from below,
fix $\delta>0$ arbitrarily and consider the function
\begin{equation*}
g(F):=f^{qc}(F)+\delta\vert F\vert ^2.
\end{equation*}
Note that $g$ is quasiconvex and satisfies
\[
\delta\vert F\vert ^2\leq g(F)\leq c(1+\vert F\vert ^2),\qquad
\mbox{for all }F\in\Rd. 
\]
Then, by Lemma~\ref{lemma:kristensen},
there exists a nondecreasing sequence of quasiconvex functions $\psi_k\leq g$, 
converging to $g$ pointwise and such that 
$\psi_k(F)=a_k\vert F\vert + b_k$ for all $\vert F\vert\geq r_k$, 
for some $a_k, r_k>0$ and $b_k\in\mathbb{R}$.
Now, observe that for every $k$ there exists $\hat j=\hat j(\delta,k)$
such that
\begin{equation}\label{eq:f+delta}
f_{\ep_j}(F)+\delta|F|^2\geq\psi_k(F)-\frac 1k,\quad\mbox{for every }F\in\Rd,\,
j\geq\hat j.
\end{equation}
This is because for every $k$ there exists $\hat c=\hat c(\delta,k)\geq r_k$ such that
\begin{equation*}
\delta|F|^2\geq a_k|F|+b_k-\frac1k
            =\psi_k(F)-\frac1k,\quad\qquad\mbox{for every }|F|>\hat c.
\end{equation*}
Moreover, since $f_{\ep_j}\to f$ uniformly on $\Om_k:=\{F\in\Rd:|F|\leq\hat c\}$, 
there exists $\hat j=\hat j(\delta,k)$ such that
\begin{equation*}
f_{\ep_j}(F)+\delta|F|^2\geq f(F)+\delta|F|^2-\frac1k\geq g^{qc}(F)-\frac1k
\geq\psi_k(F)-\frac1k,   
\end{equation*}
for every $F\in\Om_k$ and all $j\geq\hat j$.

Using (\ref{eq:E_jf_j}) and (\ref{eq:f+delta}) we can then write
\begin{equation}\label{eq:fine-1}
\mathscr E_{\ep_j}(u_j)\geq\int_{B_j}
 \left\{\psi_k\left(\na u_j+\frac{\zeta(\ep_j\na u_j)}{\ep_j}\right)
     -\frac 1k-\delta\left|\na u_j+\frac{\zeta(\ep_j\na u_j)}{\ep_j}\right|^2\right\}\,dx.
\end{equation}     
Focusing on the first term on the right-hand side of (\ref{eq:fine-1}), note that
\begin{multline}\label{eq:fine0}
\int_{B_j}\psi_k\left(\na u_j+\frac{\zeta(\ep_j\na u_j)}{\ep_j}\right)\,dx\\
   \geq\int_{B_j}\psi_k(\na u_j)\,dx-
   \int_{B_j}\left|\psi_k(\na u_j)-\psi_k\left(\na u_j+\frac{\zeta(\ep_j\na u_j)}{\ep_j}\right)\right|\,dx,
\end{multline}
and that the second summand on the right-hand side of (\ref{eq:fine0}) is bounded by
\begin{multline*}
\int_{\{|\na u_j|\leq M\}}
\left|\psi_k(\na u_j)-\psi_k\left(\na u_j+\frac{\zeta(\ep_j\na u_j)}{\ep_j}\right)\right|\,dx\\
+\int_{B_j\cap\{|\na u_j|> M\}}\left(|\psi_k(\na u_j)|
 +\left|\psi_k\left(\na u_j+\frac{\zeta(\ep_j\na u_j)}{\ep_j}\right)\right|\right)\,dx,
\end{multline*}
for a fixed $M>0$. Therefore, on the one hand, 
since each $\psi_k$ is quasiconvex and hence 
locally Lipschitz (see \cite[Theorem 2.31]{Dacorogna}),
we have that
\begin{align}\label{eq:fine1}
\int_{\{|\na u_j|\leq M\}}
\bigg|\psi_k(\na u_j)-\psi_k\bigg(\na u_j+&\frac{\zeta(\ep_j\na u_j)}{\ep_j}\bigg)\bigg|\,dx\nonumber\\
&\leq c\,\int_{\{|\na u_j|\leq M\}}\left|\frac{\zeta(\ep_j\na u_j)}{\ep_j}\right|\,dx\nonumber\\
&\leq c\,\ep_j\int_{\{|\na u_j|\leq M\}}|\na u_j|^2\,dx\leq c\,\ep_j M^2,
\end{align}
where in the second inequality we have used (\ref{eq:prop_h}) and $c=c(k,M)$.
On the other hand, since $\psi_k\leq\tilde a_k|\cdot|+\tilde b_k$ for suitable
constants $\tilde a_k$, $\tilde b_k$, then
\begin{multline}\label{eq:fine2}
\int_{B_j\cap\{|\na u_j|> M\}}\left(|\psi_k(\na u_j)|
 +\left|\psi_k\left(\na u_j+\frac{\zeta(\ep_j\na u_j)}{\ep_j}\right)\right|\right)\,dx\\
\leq\int_{\{|\na u_j|> M\}}\big[(2+c)\tilde a_k|\na u_j|+2\tilde b_k\big]\,dx,
\end{multline}
where we have also used (\ref{eq:prop_h}).
Now, by the equiintegrability of $\{|\na u_j|\}$, we can choose $M=M_k$ such that
\begin{equation*}
\int_{\{|\na u_j|> M_k\}}\big[(2+c)\tilde a_k|\na u_j|+2\tilde b_k\big]\,dx\leq\frac1{2k}. 
\end{equation*}
Also, up to a bigger $\hat j$, we can suppose that
$c\,\ep_j M^2_k\leq1/(2k)$ for every $j\geq\hat j$, 
so that (\ref{eq:fine0}) and inequalities (\ref{eq:fine1}) and (\ref{eq:fine2})
(with $M_k$ in place of $M$) yield
\begin{equation}\label{eq:fine3}
\int_{B_j}\psi_k\left(\na u_j+\frac{\zeta(\ep_j\na u_j)}{\ep_j}\right)\,dx\\
   \geq\int_{B_j}\psi_k(\na u_j)\,dx-\frac1k.
\end{equation}
Going back to (\ref{eq:fine-1}), observe that
\begin{equation}\label{eq:fine4}
\int_{B_j}\delta\left|\na u_j+\frac{\zeta(\ep_j\na u_j)}{\ep_j}\right|^2\,dx
\leq c\,\delta\int_{\Om}1_{B_j}|\na u_j|^2\,dx\leq c\,\delta,
\end{equation}
in view of (\ref{eq:prop_h}) and the 
boundedness of $\{1_{B_j}\na u_j\}$ in $L^2(\Om,\Rd)$.
Inequalities (\ref{eq:fine-1}), (\ref{eq:fine3}), and (\ref{eq:fine4}) give
\begin{equation}\label{eq:fine5}
\mathscr E_{\ep_j}(u_j)\geq\int_{B_j}\psi_k(\na u_j)\,dx-\frac ck-c\,\delta.
\end{equation}
Concentrating on the term involving $\psi_k$, let us write
\begin{equation}\label{eq:fine6}
\int_{B_j}\psi_k(\nabla u_j)\,dx = \int_{\Omega}\psi_k(\nabla u_j)\,dx
                   -\int_{B^{c}_{j}}\psi_k(\nabla u_j)\,dx,
\end{equation}
and note that
\[
\int_{B^{c}_{j}}\psi_k(\nabla u_j)\,dx
     \leq\tilde a_k\int_{B^{c}_{j}}\vert\nabla u_j\vert\,dx 
   +\tilde b_k\vert B^{c}_{j}\vert
\leq\tilde a_k\vert B^{c}_{j}\vert ^{\frac{p-1}{p}}\Vert\nabla u_j\Vert _p 
+\tilde b_k\vert B^{c}_{j}\vert,
\]
where we have used H\"older's inequality. Substituting into (\ref{eq:fine6}),
since $\{u_j\}$ is uniformly bounded in $W^{1,p}(\Omega,\mathbb{R}^d)$, 
inequality (\ref{eq:fine5}) now reads
\begin{equation}\label{eq:fine7}
\mathscr E_{\ep_j}(u_j)\geq
\int_{\Omega}\psi_k(\nabla u_j)\,dx-c\,\tilde a_k\vert B^{c}_{j}\vert^{\frac{p-1}{p}}
-\tilde b_k\vert B^{c}_{j}\vert -\frac ck-c\,\delta.
\end{equation}
We may now take the $\liminf$ over $j$ in (\ref{eq:fine7}) to infer that, 
since $\vert B^{c}_{j}\vert\rightarrow 0$ as $j\rightarrow\infty$,
\begin{equation}\label{eq:fine8}
\liminf_{j} \mathscr E_{\ep_j}(u_j)
   \geq\int_{\Omega}\psi_k(\nabla u)\,dx -\frac ck- c\,\delta,
\end{equation}
where we have used the fact that for each $k$, $\psi_k$ is a 
quasiconvex function satisfying a linear growth, and therefore 
the functional $w\mapsto\int_{\Omega}\psi_k(\nabla w)\,dx$ is sequentially weakly lower 
semicontinuous in $W^{1,1}(\Omega,\mathbb{R}^d)$.

Having eliminated $j$ in \eqref{eq:fine8}, 
we may now take the limit in $k$ and we can deduce, by monotone convergence, that
\begin{equation*}
\liminf_{j} \mathscr E_{\ep_j}(u_j)\geq 
\int_{\Omega}f^{qc}(\nabla u)\,dx + \delta\int_{\Omega}\vert\nabla u\vert ^2\,dx - c\,\delta.
\end{equation*}
However, $\nabla u\in L^{2}(\Omega,\Rd)$ by Lemma~\ref{lemma:bj} (ii) and $u\in W^{1,p}_h$. 
Therefore $u\in W^{1,2}_h$ and in turn
$ \int_{\Omega}f^{qc}(\nabla u)\,dx=\overline{\mathscr E}(u)$. 
Finally, since $\delta$ is arbitrary,
the $\Gamma$-$\liminf$ inequality follows by letting $\delta\to0$.

For the $\Gamma$-$\limsup$ inequality, we need to establish that 
$\mathscr E''(u)\leq\overline{\mathscr E}(u)$ for all 
$u\in W^{1,p}(\Omega,\mathbb{R}^d)$, where we recall that
\[
\mathscr E''(u):=\Gamma\mbox{-}\limsup_{j\rightarrow\infty}\mathscr E_{\ep_j}(u)=\inf\left\{\limsup_{j\rightarrow\infty}\mathscr E_{\ep_j}(u_j): u_j\rightharpoonup u\mbox{ in $W^{1,p}(\Omega,\mathbb{R}^d)$}\right\}.
\]
In fact, it suffices to prove that $\mathscr E''(u)\leq\overline{\mathscr E}(u)$ for all $u\in W^{1,2}_h$ as otherwise the result follows trivially.

First, note that the functional $\overline{\mathscr E}$ is continuous on $W^{1,2}_h$ 
in the strong topology of $W^{1,2}(\Omega,\mathbb{R}^d)$.
Indeed, letting $u_k, u\in W^{1,2}_h$, 
\begin{align}
\vert\overline{\mathscr E}(u_k)-\overline{\mathscr E}(u)\vert 
&\leq\int_{\Omega}\vert f^{qc}(\nabla u_k)-f^{qc}(\nabla u)\vert\,dx\nonumber\\
&\leq c\,\int_{\Omega}(1+\vert\nabla u_k\vert +\vert\nabla u\vert)\vert\nabla u_k 
    -\nabla u\vert\,dx\label{eq:stronglycts}\\
&\leq c\,\Vert\nabla u_k-\nabla u \Vert_{2},\nonumber
\end{align}
where the second inequality follows from \cite[Proposition 2.32]{Dacorogna}.
This continuity property allows us to work on the smaller space $W^{1,\infty}(\Omega,\mathbb{R}^d)$. 

For convenience, let us define the functional $\mathscr F:W^{1,\infty}\rightarrow\overline{\mathbb{R}}$ given by
\begin{equation*}
\label{eq:Flimit}
\mathscr F(u):=\left\{\begin{array}{ll}\int_{\Omega} f(\nabla u)\,dx, & u\in W^{1,\infty}_h\\
\,&\,\\
+\infty ,& u\in W^{1,\infty}(\Omega,\mathbb{R}^d)\setminus W^{1,\infty}_h .\end{array}\right. 
\end{equation*}
By Theorem \ref{thm:relaxation}, the sequential weak* 
lower semicontinuous envelope of $\mathscr F$ in $W^{1,\infty}(\Omega,\mathbb{R}^d)$ 
is given by
\begin{equation*}
\label{eq:Frel}
\overline{\mathscr F}(u):=\left\{\begin{array}{ll}\int_{\Omega} f^{qc}(\nabla u)\,dx, & u\in W^{1,\infty}_h\\
\,&\,\\
+\infty ,& u\in W^{1,\infty}(\Omega,\mathbb{R}^d)\setminus W^{1,\infty}_h .\end{array}\right. 
\end{equation*}
We now claim that 
$\mathscr E''(u)\leq\overline{\mathscr F}(u)$ for all 
$u\in W^{1,\infty}(\Omega,\mathbb{R}^d)$. We first establish that
\begin{equation}
\label{eq:gamma7}
\mathscr E''(u)\leq\mathscr F(u),\quad\mbox{for all $u\in W^{1,\infty}(\Omega,\mathbb{R}^d)$}.
\end{equation}
Indeed, we may assume that $u\in W^{1,\infty}_h$ as otherwise there is nothing to prove. 
Note that, for $u\in W^{1,\infty}_h$, the range of $\nabla u$ is compact and, 
therefore, $f_{\ep_j}(\nabla u(x))\rightarrow f(\nabla u(x))$ uniformly,
where $f_\ep$, $f$ are defined in \eqref{eq:fep}-\eqref{eq:fonly}. Then,
\begin{align*}
  \mathscr E''(u)&:=\inf\left\{\limsup_{j\to \infty}\mathscr E_{\ep_j}(u_j): u_j\rightharpoonup u
\mbox{ in $W^{1,p}(\Omega,\mathbb{R}^d)$}\right\}   
    \leq  \limsup_{j\to \infty}\mathscr E_{\ep_j}(u) \\
&= \lim_{j\to\infty}\frac{1}{\ep^{2}_{j}}\int_{\Omega}W_{\ep_j}(I+\ep_j\nabla u(x))\,dx
=\int_{\Omega}f(\nabla u(x))\,dx = \mathscr F(u),
\end{align*}
for every $u\in W^{1,\infty}_h$, proving (\ref{eq:gamma7}). 
Next, note that $\mathscr E''$ is sequentially weakly 
lower semicontinuous in $W^{1,p}(\Omega,\mathbb{R}^d)$,
since it is an upper $\Gamma$ limit, 
and therefore it is also sequentially weak* lower 
semicontinuous in $W^{1,\infty}(\Omega,\mathbb{R}^d)$. But $\overline{\mathscr F}$ is the largest 
functional below $\mathscr F$ enjoying this lower semicontinuity property and, hence,
\begin{equation}
\label{eq:gamma8}
\mathscr E''(u)\leq\overline{\mathscr F}(u),
\quad\mbox{for all $u\in W^{1,\infty}(\Omega,\mathbb{R}^d)$},
\end{equation}
establishing our claim. We now show that $\mathscr E''(u)\leq\overline{\mathscr E}(u)$ 
for all $u\in W^{1,2}_h$ proving the $\Gamma$-$\limsup$ inequality. 
Let $u\in W^{1,2}_h$ and, by \cite[Proposition A.2]{AgDalDe},
consider a sequence $u_k\in W^{1,\infty}_h$ such that $u_k\rightarrow u$ 
(strongly) in $W^{1,2}(\Omega,\mathbb{R}^d)$; 
in particular, $u_k\rightharpoonup u$ in $W^{1,p}(\Omega,\mathbb{R}^d)$. 
Then, by the lower semicontinuity of $\mathscr E''$, 
and by \eqref{eq:stronglycts} and \eqref{eq:gamma8}, 
\begin{align*}
  \mathscr E''(u)&\leq \liminf_{k\to\infty}\mathscr E''(u_k)
\leq  \liminf_{k\to\infty}\overline{\mathscr F}(u_k)\\
&= \liminf_{k\to\infty}\int_{\Omega}f^{qc}(\nabla u_k(x))\,dx
=\int_{\Omega}f^{qc}(\nabla u(x))\,dx =\overline{\mathscr E}(u),
\end{align*}
for every $u\in W^{1,2}_h$. This completes the proof. 
\end{proof}

\begin{proof}[Proof of Theorem \ref{main_thm}]
Define the functionals $\mathscr G_\ep$, $\overline{\mathscr G}:W^{1,p}(\Om,\RR^d)\to(-\infty,\infty]$ as
\[
\mathscr G_\ep:=\mathscr E_\ep-\mathscr L,\qquad\qquad
\overline{\mathscr G}:=\overline{\mathscr E}-\mathscr L,
\]
so that, by the definition of $\mathscr E_\ep$ and $\overline{\mathscr E}$, the hypothesis of the theorem
can be rewritten as
\[
\lim_{\ep\to0}\mathscr G_\ep(u_\ep)=\lim_{\ep\to0}m_\ep,\qquad\qquad m_\ep:=\inf_{W^{1,p}_h}\mathscr G_\ep.
\]
By Theorem \ref{thm:gammaconvergence} and \cite[Proposition 6.21]{Dal_Maso_book},
we infer that 
\begin{equation}\label{eq:Gammacvg_G}
\{\mathscr G_\ep\}\ \ \ \Gamma\mbox{-converges to }\ \ \overline{\mathscr G}\ \ 
\mbox{ in the weak topology of }\ W^{1,p}(\Om,\RR^d).
\end{equation}
Moreover, from Proposition \ref{prop:compactness} we deduce that $\{\mathscr G_\ep\}$ is a weakly-equicoercive 
sequence of functionals in $W^{1,p}(\Om,\RR^d)$, because $p>1$ and,
in view of Poincar\'e's inequality, 
\begin{align*}
\|u\|_{W^{1,p}}^p&\leq c(1+\mathscr G_\ep(u)+\mathscr L(u))\\
                 &\leq c(1+\mathscr G_\ep(u)+\|u\|_{W^{1,p}}),\quad\mbox{for every }u\in W^{1,p}_h.
\end{align*}
By standard $\Gamma$-convergence arguments, the equicoercivity of $\{\mathscr G\ep\}$
and convergence (\ref{eq:Gammacvg_G}) ensure that $m_\ep\to m$ (see \cite[Theorem 7.8]{Dal_Maso_book}).
Another standard argument then shows that, up to a subsequence, $u_\ep\rightharpoonup u$ weakly  
in $W^{1,p}(\Om,\RR^d)$, where $u$ is a minimizer of $\overline{\mathscr G}$. 
\end{proof}

\begin{proof}[Proof of Corollary \ref{cor:lowen}]
By using the notation (\ref{eq:Eepsilon})-(\ref{eq:Erel}), we can rewrite the 
hypothesis of the corollary as 
$\liminf_{\ep\to0}\mathscr E_{\ep}(u_{\ep})=0$.
Therefore, from Proposition~\ref{prop:compactness}, 
we obtain that, up to a subsequence,
$u_{\ep}\rightharpoonup u$ in $W^{1,p}(\Om,\RR^d)$, 
for some $u\in W^{1,p}(\Om,\RR^d)$. 
By Theorem \ref{thm:gammaconvergence},
we have that indeed $u\in W^{1,2}_h$ (recall that here $\pa_D\Om=\pa\Om$ and $h(x)=Fx$ on $\pa\Om$) and
\[
0\leq\overline{\mathscr E}(u)\leq\liminf_{\ep\to0}\mathscr E_{\ep}(u_{\ep})=0,
\]
so that ${\rm sym}\,(\na u)\in\{V^{qce}=0\}$ a.e.~in $\Om$, 
since $V^{qce}({\rm sym}\,G)=f^{qc}(G)$ for every $G\in\Rd$, by Proposition \ref{prop:qcqce}.
Now since $0\leq f^{qc}\leq c(1+|\cdot|^2)$, $f^{qc}$ is $W^{1,2}$-quasiconvex 
by \cite{BallMurat} (see Appendix for definitions), so that
\[
0\leq V^{qce}({\rm sym}\,F)=f^{qc}(F)\leq\fint_{\Om}f^{qc}(\na u)dx
=\fint_{\Om}V^{qce}(e(u))dx=0,
\]
and in turn $V^{qce}({\rm sym}\,F)=0$. 
Now, to see that $\{V^{qce}=0\}\subseteq Q^e_2\mathcal M$,
we note that
\begin{equation*}
V(E):=\lim_{\ep\to0}\frac{W_{\ep}(I+\ep E)}{\ep^2}
\geq\lim_{\ep\to0}\frac{g_p({\rm dist}(I+\ep E,\mathcal U_{\ep}))}{\ep^2}
=\lim_{\ep\to0}\frac{{\rm dist}^2(I+\ep E,\mathcal U_{\ep})}{2\ep^2}.
\end{equation*}
Thus, in view of Lemma \ref{lem:dist} below, we have that $V(E)\geq\frac12{\rm dist}^2(E,\mathcal M)$
for every $E\in\Rds$. This in particular implies that
$$
V^{qce}\geq\frac12\left({\rm dist}^2(\cdot,\mathcal M)\right)^{qce},
$$
so that, by definition of $Q^e_2\mathcal M$, if $E\in\{V^{qce}=0\}$ then $E\in Q^e_2\mathcal M$.
This concludes the proof.
\end{proof}

\begin{lemma}\label{lem:dist}
If $E\in \Rds$ and $\mathcal U_\ep$ is given by (\ref{eq:intro_wells}), then
\begin{equation}
  \label{eq:limd}
  \lim_{\ep \to 0}\frac{1}{\ep^2}{\rm dist}^2(I+\ep E, \, \mathcal U_\ep)
={\rm dist}^2(E,\mathcal M).
\end{equation}
\end{lemma}

\begin{proof}For a fixed $E\in \Rds$, we have that
  \begin{align*}
    {\rm dist} (I+\ep E, \, \mathcal U_\ep) 
& \leq \min_{U\in\mathcal M} |I+\ep E  -(I +\ep U+o(\ep))|\\
& \leq \ep \min_{U\in\mathcal M}|E-U| + o(\ep).
  \end{align*}
Hence,
\begin{equation}
  \label{eq:d2u}
  \limsup_{\ep\to0}\frac{1}{\ep^2}{\rm dist}^2(I+\ep E, \, \mathcal U_\ep)\leq
{\rm dist}^2(E,\mathcal M).
\end{equation}
Now, let $R_{\ep}\in SO(d)$ and let $U_{\ep}\in\mathcal M$ such that
the distance between $I+\ep E$ and $\mathcal U_\ep$ is achieved for
$R=R_{\ep}$ and $U = U_{\ep}$. Since $SO(d)$ and $\mathcal M$ are bounded sets,
we have that, up to a subsequence, $R_{\ep}\to\hat R$ and $U_{\ep}\to\hat U$. Also,
\begin{align*}
  {\rm dist}(I+\ep E, \, \mathcal U_\ep) &= 
|I+\ep E -R_{\ep}(I+\ep U_{\ep}+o(\ep))|\\
& \geq |I -R_{\ep}| - |\ep E -\ep R_{\ep}U_{\ep}+o(\ep)|,
\end{align*}
from which we deduce that $\hat R=I$. Indeed, if $\hat R\neq I$, then by \eqref{eq:d2u}
\[
c\geq \frac{1}{\ep}{\rm dist}(I+\ep E, \, \mathcal U_\ep)
\geq\left\{\frac{1}{\ep}|I -R_{\ep}| -|E - R_{\ep}U_{\ep}+o(1)|\right\}\to \infty,
\]
which is absurd. Now, since $R_{\ep}\to I$ and $\Rdsk:=\{A\in\Rd \,:\, A=-A^T\}$ is the
tangent space to $SO(d)$ at $I$, we have that
\[
\lim_{\ep \to 0}\frac{1}{\ep}(R_{\ep}-I) = A,\qquad\mbox{for some }A\in\Rdsk.
\]
Hence,
\begin{align}
 \liminf_{\ep \to 0} \frac{1}{\ep^2} 
{\rm dist}^2 (I+\ep E,\, \mathcal U_\ep)   
& = \lim_{\ep \to 0} \left|\frac{1}{\ep}(I-R_{\ep})
+ E-R_{\ep}U_{\ep}\right|^2\nonumber\\
&=|A +  E - \hat U|^2 
\geq |A|^2 +  \min_{U\in \mathcal M}|E - U|^2,\label{eq:d2l}
\end{align}
where in the last passage we have also used the fact that $A$ is orthogonal to $E-\hat U\in \Rds$.
From \eqref{eq:d2u} and \eqref{eq:d2l}, we have
\begin{align*}
 |A|^2 +{\rm dist}^2(E,\mathcal M)
&\leq\liminf_{\ep \to 0} \frac{1}{\ep^2}{\rm dist}^2 (I+\ep E,\, \mathcal U_\ep) \\
&\leq\limsup_{\ep \to 0} \frac{1}{\ep^2}{\rm dist}^2 (I+\ep E,\, \mathcal U_\ep) 
\leq{\rm dist}^2(E,\mathcal M), 
\end{align*}
which implies $A = 0$ and, in turn, implies \eqref{eq:limd}.
\end{proof}

\begin{remark}
  A simple case of Lemma \ref{lem:dist} is when $\mathcal M$ contains
only the zero matrix, 
so that \eqref{eq:limd} yields 
\begin{equation}
  \label{eq:limdSOd}
  \lim_{\ep \to 0}\frac{1}{\ep^2}{\rm dist}^2(I+\ep E, \, SO(d))
= |E|^2.
\end{equation}
\end{remark}


\section{Young measure representation and strong convergence}\label{sec:equiint}

In this section we wish to discuss two improvements of our results,
in the spirit of \cite[Corollary 2.5 and Theorem 3.1]{Schmidt}, that is
a Young measure representation of the limiting functional $\overline{\mathscr E}$ and
the strong convergence of recovery sequences.
 
Recall that given a sequence $\{u_j\}\subset W^{1,p}(\Om,\RR^d)$, the set $B_j\subseteq\Om$ is defined as
\[
B_j:=\left\{x\in\Omega : \vert\nabla u_j(x)\vert\leq\frac{1}{\sqrt{\ep_{j}}}\right\},
\]
and
\[
v_j(x)=x+\ep_ju_j(x),
\]
where $\{\ep_j\}$ is a vanishing sequence. 
We start with a lemma, which will enter into the following discussion. 

\begin{lemma}
\label{lemma:equiint}
Let $u\in W^{1,2}_h$ and suppose that $\{u_j\}$ is a recovery sequence for $u$. Then
\begin{equation*}
\left\{1_{B_j}\frac{{\rm dist}^2(\nabla v_j, SO(d))}{\ep^{2}_{j}}\right\}
\end{equation*}
is equiintegrable.
\end{lemma}

\begin{proof}
For notational convenience, let $\rho_j:=1_{B_j}\ep^{-2}_{j}{\rm dist}^2(\nabla v_j, SO(d))$ 
and suppose for contradiction that $\{\rho_j\}$ is not equiintegrable, 
i.e.~there exists some $\alpha>0$ such that for all $k$ there exists $j_k$ with 
\begin{equation*}
\int\limits_{\left\{x\in\Omega :\rho_{j_k}(x)\geq k\right\}}\rho_{j_k}\,dx\geq\alpha.
\end{equation*}
Fixing any $M>0$, we have that 
$
\int_{\left\{x\in\Omega :\rho_{j_k}(x)\geq M\right\}}\rho_{j_k}\,dx\geq\alpha.
$ 
In particular, up to passing to the subsequence $\rho_{j_k}$, we may assume
\begin{equation}
\label{eq:equi1}
\liminf_{j\rightarrow\infty}
\int\limits_{\left\{x\in\Omega : \rho_{j}\geq M\right\}}\rho_{j}\,dx\geq\alpha,
\qquad\mbox{ for all $M>0$.}
\end{equation}
Since $\{u_j\}$ is a recovery sequence for $u\in W^{1,2}_h$, we have that
$\overline{\mathscr E}(u)=\int_{\Omega}f^{qc}(\nabla u)\,dx=\lim_{j}\mathscr E_{\ep_j}(u_j)$.
However, for $j$ large enough
\begin{equation}\label{eq:equi2}
\mathscr E_{\ep_j}(u_j)
\geq\int\limits_{\left\{x\in B_j : \vert\nabla u_j\vert <M\right\}}\frac{W_{\ep_j}(\nabla v_j)}{\ep^{2}_{j}}\,dx
+\quad c\!\!\!\!\!\!\!\!\!\!\!\!\int\limits_{\left\{x\in B_j : \vert\nabla u_j\vert\geq M\right\}}\frac{{\rm dist}^2(\nabla v_j, \mathcal U_{\ep_j})}{2\ep^2_j}\,dx,
\end{equation}
because $W_{\ep_j}(\na v_j)\geq c\, g_p({\rm dist}(\na v_j,\mathcal U_{\ep_j}))$ and
${\rm dist}(\na v_j,\mathcal U_{\ep_j})\leq1$ for a.e.~$x\in B_j$ and every $j$ large enough.
Next, note that for a.e.~$x$, there exists a rotation $R_j(x)$ and a matrix $U_j(x)\in\mathcal{M}$ such that
${\rm dist}(\nabla v_j(x),\mathcal{U}_{\ep_j})\geq\vert\nabla v_j(x)-R_j(x) \vert -\vert\ep_{j}U_j(x)+o(\ep_{j})|$.
But then,
\begin{equation}
\label{eq:equi3}
{\rm dist}^2(\nabla v_j(x),\mathcal{U}_{\ep_j})
\geq \frac{{\rm dist}^2(\nabla v_j(x), SO(d))}{2} - c\,\ep^{2}_{j}
\end{equation}
and combining (\ref{eq:equi2}) with (\ref{eq:equi3}) we deduce that
\begin{multline}\label{eq:equi4}
\mathscr E_{\ep_j}(u_j)
\geq\int\limits_{\left\{x\in B_j : \vert\nabla u_j\vert <M\right\}}\frac{W_{\ep_j}(\nabla v_j)}{\ep^{2}_{j}}\,dx\\
+\quad c\!\!\!\!\!\!\!\!\!\!\!\!\int\limits_{\left\{x\in\Omega : \vert\nabla u_j\vert\geq M\right\}}\frac{\rho_j}{4}\,dx
-c\,\vert\left\{x\in B_j : \vert\nabla u_j\vert\geq M\right\}\vert .
\end{multline}
Also, note that
$\vert\nabla u_j\vert =\ep_j^{-1}\vert\nabla v_j-I\vert\geq\ep_j^{-1}{\rm dist}(\nabla v_j,SO(d))$, so that
\[
\left\{x\in B_j : \rho_j(x)\geq M^2\right\}\subseteq
\left\{x\in B_j : \vert\nabla u_j\vert\geq M\right\},
\]
and that $\vert\left\{x\in B_j : \vert\nabla u_j\vert\geq M\right\}\vert\leq c/M$. 
Hence, (\ref{eq:equi4}) becomes
\begin{equation}\label{eq:equi7}
\mathscr E_{\ep_j}(u_j)\geq\int\limits_{\left\{x\in B_j : 
\vert\nabla u_j\vert <M\right\}}\frac{W_{\ep_j}(\nabla v_j)}{\ep^{2}_{j}}\,dx
+\quad c\!\!\!\!\!\!\!\!\!\!\!\!\int\limits_{\left\{x\in\Omega : \rho_j(x)\geq M^2\right\}}\frac{\rho_j}{4}\,dx - \frac cM.
\end{equation}
Let us now consider the first term on the right side of (\ref{eq:equi7}). 
Proceeding as in the proof of the $\Gamma$-$\liminf$ inequality in Theorem~\ref{thm:gammaconvergence}
(to which we refer for the notation), 
we infer that for some arbitrary $\delta$,
\begin{equation}\label{eq:equi8}
\int\limits_{\left\{x\in B_j : \vert\nabla u_j\vert <M\right\}}
\frac{W_{\ep_j}(\nabla v_j)}{\ep^{2}_{j}}\,dx
\geq\int\limits_{\left\{x\in B_j : \vert\nabla u_j\vert <M\right\}}\psi_k(\nabla u_j)\,dx
-\frac ck-c\,\delta,
\end{equation}
for every $j\geq\hat j=\hat j(\delta,k)$.
Also, we have that
\begin{align}
\int\limits_{\left\{x\in B_j : \vert\nabla u_j\vert <M\right\}}\psi_k(\nabla u_j)\,dx
&\geq\int_{B_j}\psi_k(\nabla u_j)\,dx\nonumber\\
&-\int\limits_{\left\{x\in B_j : \vert\nabla u_j\vert\geq M\right\}}
(\tilde a_k\vert\nabla u_j\vert + \tilde b_k)\,dx\nonumber\\
&\geq\int_{B_j}\psi_k(\nabla u_j)\,dx - \frac c{M}\tilde a_k - \frac c{M}\tilde b_k.
\label{eq:equi9}
\end{align}
By inequalities (\ref{eq:equi8}) and (\ref{eq:equi9}), we can rewrite (\ref{eq:equi7}) as 
\begin{multline*}\label{eq:equi11}
\mathscr E_{\ep_j}(u_j)\geq
\int_{B_j}\psi_k(\nabla u_j)\,dx-\frac c{M}\tilde a_k - \frac c{M}\tilde b_k
-\frac ck-c\,\delta\\
+\quad c\!\!\!\!\!\!\!\!\!\!\!\!\int\limits_{\left\{x\in\Omega : \rho_j(x)\geq M^2\right\}}\frac{\rho_j}4\,dx - \frac cM.
\end{multline*}
Taking the limit as $j\rightarrow\infty$ and using (\ref{eq:equi1}) gives 
\begin{equation*}
\int_{\Omega}f^{qc}(\nabla u)\,dx\geq \int_{\Omega}\psi_k(\nabla u)\,dx+c\,\alpha - 
\frac c{M}\tilde a_k - \frac c{M}\tilde b_k-\frac ck-c\,\delta-\frac cM,
\end{equation*}
and letting $M\rightarrow\infty$, we obtain
\begin{equation}
\label{eq:equi13}
\int_{\Omega}f^{qc}(\na u)\,dx\geq \int_{\Omega}\psi_k(\na u)\,dx+c\,\alpha
-\frac ck -c\,\delta.
\end{equation}
Next, taking the limit $k\rightarrow\infty$ in (\ref{eq:equi13}), by monotone convergence we get that
\begin{equation*}
\int_{\Omega}f^{qc}(\nabla u)\,dx\geq \int_{\Omega}f^{qc}(\nabla u)\,dx+
\delta\int_{\Omega}\vert\nabla u\vert ^2dx+c\,\alpha -c\,\delta,
\end{equation*}
and by the arbitrariness of $\delta$ we deduce that
\begin{equation*}
\int_{\Omega}f^{qc}(\nabla u)\,dx\geq \int_{\Omega}f^{qc}(\nabla u)\,dx+c\,\alpha,
\end{equation*}
which is absurd since $c\,\alpha>0$. This contradiction concludes the proof.
\end{proof}

The next result says that 
the Young measure representation of the limiting functional $\overline{\mathscr E}$
at a point $u\in W_h^{1,2}$ can be obtained through the Young measure 
generated by the gradients of a recovery sequence for $u$.

\begin{proposition}\label{prop:YM_rep}
Under the conditions of Theorem \ref{main_thm}, if $\{u_j\}$ is a recovery sequence 
for $u\in W^{1,2}_h$. Then we have the representation
\begin{equation*}
\overline{\mathscr E}(u)=\int_{\Om}\int_{\Rd}f(F)d\nu_x(F)dx,
\end{equation*}
where $(\nu_x)_{x\in \Om}$ is the $p$-gradient Young measure generated by the sequence $\{\na u_j\}$.
\end{proposition}

Before proving this result, some comments are in order.
When the growth condition (\ref{eq:coercivity}) with $p=2$ is assumed, 
proving the Young measure representation can be reduced to showing 
that a recovery sequence $\{u_j\}$ is indeed a relaxing sequence for the 
limiting functional. In turn, this amounts to
the equiintegrability of $\{|\na u_j|^2\}$, which can 
be deduced from the rigidity results of \cite{CoDoMu}
and the equiintegrability of 
$\{{\rm dist}^2(\na v_j,SO(d))/\ep_j^2\}$ (see \cite[Lemma 4.2]{Schmidt}).
For $1<p<2$, one cannot hope to prove the latter; instead, the previous lemma says that 
the equiintegrability of 
$\{1_{B_j}{\rm dist}^2(\na v_j,SO(d))/\ep_j^2\}$ holds true. Nevertheless, this cannot be used
to prove the equiintegrability of the sequence $\{|1_{B_j}\na u_j|^2\}$ 
(which could play the role of $\{|\nabla u_j|^2\}$)
via a straightforward application of \cite{CoDoMu}. 
The reason is essentially that $1_{B_j}\na u_j$ is not 
itself a gradient and we cannot apply the rigidity results, 
at least not in any obvious way.
Note, however, that recovery sequences are indeed $p$-equiintegrable
as the following remark points out. 

\begin{remark}\label{remark:p-equiintegrable}
As done in \cite{AgDalDe}, it is possible to prove that the sequence
\begin{equation*}
\left\{\frac{{\rm dist}^p(\nabla v_j,SO(d))}{\ep^{p}_{j}}\right\}
\end{equation*}
is equiintegrable. An application of \cite{CoDoMu} then gives the 
equiintegrability of $\left\{\vert\nabla u_j\vert ^p\right\}$.
However, this is not enough for the Young measure representation due to the quadratic growth of $f$.
\end{remark}

In view of the previous discussion, a different strategy needs to be sought in order
to prove the Young measure representation.
The idea is to show that, given a recovery sequence $\{u_j\}$, 
one can construct a ``relaxing'' sequence for the limiting functional
via a suitable truncation of $\nabla u_j$. This is the content of the following lemma,
where the truncation operator $T_k:\Rd\to\Rd$ is defined as
\[
T_k(F):=\left\{\begin{array}{ll}
k\frac{F}{|F|}, & |F|>k\\
 & \\
F, & |F|\leq k.
\end{array}\right.
\]

\begin{lemma}\label{lemma:truncation}
Under the conditions of Theorem \ref{main_thm}, if $\{u_j\}$ is a recovery sequence 
for $u\in W^{1,2}_h$, there exists a subsequence $\{u_{j_k}\}\subset\{u_j\}$ such that
the truncations $Y_k:\Om\to\Rd$ defined as
\[
Y_k:=T_k(\nabla u_{j_k})
\]
satisfy
\begin{equation}\label{oioia}
\int_\Om f(Y_k)dx\longrightarrow\int_\Om f^{qc}(\nabla u)dx.
\end{equation}
\end{lemma}

\begin{proof}
Let $\{u_j\}$ be a recovery sequence for $u$, so that
$u_j\rightharpoonup u$ weakly in $W^{1,p}(\Om,\RR^d)$ and $\mathscr E_{\ep_j}(u_j)\to\overline{\mathscr E}(u)$,
and let $(\nu_x)_{x\in\Om}$ be the $p$-gradient Young measure
generated by $\{\na u_j\}$. We split the proof into several steps.

\emph{Step 1.} We prove that $(\nu_x)_{x\in\Om}$ is
a $2$-gradient Young measure. Indeed, since
\begin{equation*}
\left\{x\in\Omega : \nabla u_j(x)\neq 1_{B_j}\nabla u_j(x)\right\}=B^{c}_{j}
\end{equation*} 
and $\vert B^{c}_{j}\vert\rightarrow 0$, the sequences $\left\{\nabla u_j\right\}$ and 
$\left\{1_{B_j}\nabla u_j\right\}$ generate the same Young measure $(\nu_x)_{x\in\Om}$
(see e.g. \cite[Lemma 6.3]{Pedregal_book}). 
In particular, by \cite[Theorem 6.11]{Pedregal_book}, we have that
\begin{equation*}
\int_{\Om}\langle\nu_x,|\cdot|^2\rangle\,dx:=\int_{\Om}\int_{\Rd}|A|^2d\nu_x(A)\,dx
\leq\liminf_{j\to\infty}\int_{\Om}|1_{B_j}\na u_j|^2\,dx\leq c.
\end{equation*}
This bound, together with the fact that $(\nu_x)_{x\in\Om}$ is a $p$-gradient Young measure
concludes the proof of Step 1, 
in view of \cite[Corollary 1.8]{Kristensen}.

\emph{Step 2.} For every subsequence $\{u_{j_k}\}\subset\{u_j\}$, the sequence 
$\{Y_k\}$ defined as in the statement generates the same Young measure $(\nu_x)_{x\in\Om}$.
As in Step 1, this follows directly from
\[
|\{x\in\Om:Y_k(x)\neq\nabla u_{j_k}(x)\}| = |\{x\in\Om:|\nabla u_{j_k}(x)|>k\}|\leq \frac 1k \int_\Om |\nabla u_{j_k}|dx,
\]
and the boundedness of $\{\nabla u_j\}$ in $L^p(\Om,\Rd)$.

\emph{Step 3.} Here we show that there exists a subsequence $\{u_{j_k}\}\subset\{u_j\}$ such that
$\{|Y_k|^2\}$ is equiintegrable. Note that
\[
\lim_{k\to\infty} \lim_{j\to\infty}\int_\Om |T_k(\nabla u_j)|^2\,dx
= \lim_{k\to\infty}\int_\Om \langle\nu_x,|T_k(\cdot)|^2\rangle\,dx
= \int_\Om \langle\nu_x, |\cdot |^2\rangle\,dx,
\]
where the first equality follows from the fact that, for every $k$, $\{|T_k(\nabla u_j)|^2\}_j$ 
is equiintegrable and the second equality from monotone convergence. 
Hence, for a diagonal subsequence $\{j_k\}$, the convergence 
\[
\int_\Om |Y_k|^2dx=\int_\Om |T_k(\nabla u_{j_k})|^2dx\longrightarrow \int_\Om \langle\nu_x, |\cdot |^2\rangle dx
\]
holds as $k\to\infty$. Standard arguments then imply that $\{|Y_k|^2\}$ is equiintegrable, 
in view of the fact that $(\nu_x)_{x\in\Om}$ is $2$-gradient Young measure.

\emph{Step 4.} We now conclude the proof of the lemma. On the one hand, 
by \cite[Theorem 6.11]{Pedregal_book} and the standard characterization
of gradient Young measures,
\begin{equation}\label{eq:liminfYM1}
\liminf_{k\to\infty}\int_{\Om}f(Y_k)\,dx\geq\int_{\Om}\langle\nu_x,f\rangle\,dx\geq\int_{\Om}\langle\nu_x,f^{qc}
\rangle\,dx\geq\int_{\Om}f^{qc}(\na u)\,dx.
\end{equation}
We are left to show that
\[
\limsup_{k\to\infty}\int_{\Om}f(Y_k)\,dx\leq\int_{\Om}f^{qc}(\na u)\,dx.
\]
Note that, fixing $M>0$,
\begin{align*}
\int_{\Om}f^{qc}(\na u)\,dx& 
\geq\limsup_{j\to\infty} \int_{\{|\na u_j|\leq M\}} f_{\ep_j}\left(\nabla u_j+\frac{\zeta(\ep_j\na u_j)}{\ep_j}\right)\,dx\\
&=\limsup_{j\to\infty}\int_{\{|\na u_j|\leq M\}} f(\na u_j)\,dx,
\end{align*}
where $f_{\ep_j}$ and $\zeta$ are defined in \eqref{lin_lim_1} 
and \eqref{eq:zeta}, respectively, and the equality is 
due to the fact that $f_{\ep_j}\to f$ uniformly on compact 
subsets and $\zeta(\ep_j\na u_j)/\ep_j\to 0$ uniformly on $\{|\na u_j|\leq M\}$.
In particular, the previous inequality holds for the subsequence $\{u_{j_k}\}$ introduced in Step 3, leading to
\begin{equation*}
\int_{\Om}f^{qc}(\na u)\,dx\geq \limsup_{k\to\infty}\int_{\{|Y_k|\leq M\}} f(Y_k)\,dx,
\end{equation*}
where we have also used the fact that for $k$ large, $Y_k = \nabla u_{j_k}$ on 
$\{|\nabla u_{j_k}|\leq M\}$. Finally, observe that
\[
\limsup_{k\to\infty}\int_{\{|Y_k|\leq M\}} f(Y_k)\,dx\geq \limsup_{k\to\infty}\int_\Om f(Y_k)dx - 
\limsup_{k\to\infty}\int_{\{|Y_k| > M\}}f(Y_k)dx
\]
and that the second summand on the right-hand side vanishes 
as $M\to\infty$ due to Step 3 and the $2$-growth of $f$.
This concludes the proof.
\end{proof}

\begin{proof}[Proof of Proposition~\ref{prop:YM_rep}]
Let $\{u_j\}$ be a recovery sequence for $u\in W^{1,2}_h$ and $(\nu_x)_{x\in\Om}$ be the gradient Young measure associated with $\{\nabla u_j\}$. Considering the subsequence $\{u_{j_k}\}$ given by Lemma \ref{lemma:truncation}, we have that
\[
\overline{\mathscr E}(u) := \int_\Om f^{qc}(\nabla u)dx = \lim_{k\to\infty}\int_{\Om}f(Y_k)\,dx
\geq\int_{\Om}\langle\nu_x,f\rangle\,dx\geq\overline{\mathscr E}(u),
\]
where the equality follows from \eqref{oioia} and the inequalities follow from \eqref{eq:liminfYM1}.
\end{proof}

\begin{remark}
Note that Proposition \ref{prop:YM_rep} can equivalently be expressed in terms of the function $V:\Rds\to\RR$,
recalling that $V$ is obtained as the limit of $W_{\ep}(I+\ep\cdot)/\ep^2$ and $f(F):=V({\rm sym}\,F)$. 
To do so, one needs to consider the measure $\nu^{e}_{x}$ defined for a.e. $x\in\Om$ as the pushforward of
$\nu_x$ under the transformation $g: F\mapsto{\rm sym}\,F$, i.e.
\begin{equation*}
\nu^{e}_{x}(\mathcal B):=\nu_x(g^{-1}(\mathcal B)),
\end{equation*}
for all Borel subsets $\mathcal B$ of $\Rds$. 
Then, by a simple change of variables (see e.g. \cite[Section 39]{Halmos_book}), we infer that
\begin{align*}
\mathscr E(u)
&=\int_{\Omega}\int_{\mathbb{R}^{d\times d}}V(g(F))\,d\nu_x(F)\,dx\\
&=\int_{\Omega}\int_{\Rds}V(E)\,d\nu_x(g^{-1}(E))\,dx=\int_{\Omega}\int_{\Rds}V(E)\,d\nu^{e}_x(E)\,dx.
\end{align*}
\end{remark}

Turning attention to the second improvement regarding 
the strong convergence of recovery sequences, let us first recall that a function
$f:\Rd\to\RR$ is \emph{uniformly strictly quasiconvex} if
there exists a constant $\delta>0$ such that for every $F\in\Rd$
\begin{equation}\label{def:strict}
\int_{\Om}[f(F+\na\varphi)-f(F)]dx\geq\delta\int_{\Om}|\na\varphi|^2dx,
\quad\mbox{for every }\varphi\in W_0^{1,\infty}(\Om,\RR^d).
\end{equation}
We have the following proposition.

\begin{proposition}\label{prop:strong_cvg}
Under the conditions of Theorem~\ref{main_thm}, 
suppose that the limiting density $f$ is uniformly strictly quasiconvex.
If $\{u_j\}$ is a recovery sequence for $u\in W^{1,2}_h$, then $u_j\to u$ 
strongly in $W^{1,p}(\Om,\RR^d)$.
\end{proposition}

\begin{proof}
Since $f$ is uniformly strictly quasiconvex, Proposition \ref{prop:YM_rep} says that
\begin{equation}\label{eqn:to_strict}
\int_{\Omega}f(\nabla u)\,dx =\int_{\Om}\langle\nu_x,f\rangle\,dx.
\end{equation}
Also, we have that $\langle\nu_x,f\rangle\geq f(\na u(x))$ for a.e.~$x\in\Om$, by the 
characterization of gradient Young measures (see \cite[Theorem 8.14]{Pedregal_book}).
Combining this with (\ref{eqn:to_strict}), we infer that
\begin{equation}\label{eqn:to_strict_1}
\langle\nu_x,f\rangle=f(\na u(x)),\qquad\mbox{for a.e. }x\in\Om.
\end{equation}
Next, we claim that
\begin{equation}\label{eqn:to_strict_2}
\langle\nu_x,|\cdot|^2\rangle=|\na u(x)|^2,\qquad\mbox{for a.e. }x\in\Om.
\end{equation}
To see this, let us first note that, since $f$ satisfies (\ref{def:strict}) for some $\gamma>0$, then
the function $g(F):=f(F)-\tilde\gamma|F|^2$ is quasiconvex for every $0<\tilde\gamma\leq\gamma$.
This fact, together with the growth bounds $-\tilde\gamma|F|^2\leq g(F)\leq C(1+|F|^2)$
(recall that $f$ is nonnegative and that $f(F)\leq C(1+|F|^2)$ by assumption),
allows us to deduce that
$\langle\nu_x,g\rangle\geq g(\na u(x))$ for a.e.~$x\in\Om$, 
which is equivalent to
\[
\langle\nu_x,f\rangle-f(\na u(x))\geq\gamma(\langle\nu_x,|\cdot|^2\rangle-|\na u(x)|^2),
\quad\mbox{ for a.e. }x\in\Om.
\]
But the left side of this inequality is zero in view of (\ref{eqn:to_strict_1}),
so that $\langle\nu_x,|\cdot|^2\rangle\leq |\na u(x)|^2$ for a.e.~$x\in\Om$.
The convexity of the map $F\mapsto |F|^2$ then gives (\ref{eqn:to_strict_2}).
In particular, this implies that
\begin{equation}\label{eqn:to_strict_3}
\nu_x=\delta_{\na u(x)},\qquad\mbox{for a.e. }x\in\Om.
\end{equation}
This is because, since $(\nu_x)_{x\in\Om}$ is a 2-gradient Young measure, 
there exists a sequence $\{w_j\}$ such that $\na w_j \rightharpoonup \na u$ in $L^2(\Om,\Rd)$, 
$\{|\na w_j|^2\}$ is equiintegrable and $\{\na w_j\}$ generates the measure $(\nu_x)_{x\in\Om}$. 
By Young measure representation and (\ref{eqn:to_strict_2}), we infer that
\[
\lim_{j\to\infty}\int_{\Om}|\na w_j(x)|^2dx = 
\int_{\Om}\langle\nu_x, |\cdot |^2\rangle dx = \int_{\Om}|\na u(x)|^2 dx
\]
and, therefore, $\na w_j\to\na u$ strongly in $L^2(\Om,\Rd)$.
But then $\nu_x = \delta_{\na u(x)}$ a.e.~in $\Om$ by 
e.g.~\cite[Proposition 6.12]{Pedregal_book}.

We can now conclude the proof. Again by \cite[Proposition 6.12]{Pedregal_book}, 
equation (\ref{eqn:to_strict_3}) implies that, up to a subsequence, 
$\{1_{B_j}\na u_j\}$ converges to $\na u$ pointwise a.e.~in $\Om$.
Also, by Lemma~\ref{lemma:bj} (i), we have that, up to a further subsequence and for a.e.~$x\in\Om$,
$1_{B_j^c}(x)\na u_j(x)\to0$, so that
\[
\na u_j(x)=1_{B_j}(x)\na u_j(x)+1_{B_j^c}(x)\na u_j(x)\to\na u(x).
\]
Hence, in view of Remark~\ref{remark:p-equiintegrable}, 
an application of Vitali's convergence theorem concludes the proof of the proposition.
\end{proof}


\section{Appendix}


\subsection{Relaxation results}
In this section we present a version of Theorem~\ref{thm:AF} below, 
suitable for our purposes.

\begin{theorem}[Statement III.7, \cite{AcerbiFusco}]\label{thm:AF}
  Let $1\leq q \leq +\infty$ and let
$f:\Rd \to \RR$ be a continuous function
satisfying
\begin{align*}
 & 0\leq f(F) \leq c(1+|F|^q),\quad {\rm if}\,\, q\in [1,\infty)\\
  &f\,\, \mbox{is locally bounded}, \quad\qquad{\rm if}\,\,  q = \infty .
\end{align*}
Define
\begin{equation*}
  \mathscr I (u) := \int_\Omega f(\nabla u)dx,\qquad
  \overline{\mathscr I} (u) := \int_\Omega f^{qc}(\nabla u)dx,
\end{equation*}
where $f^{qc}$ is the quasiconvexification of $f$.
Then $\overline{\mathscr I}$ is 
the seq. w.l.s.c. (w*.l.s.c.) envelope of $\mathscr I$ in $W^{1,q}(\Om,\RR^d)$ ($W^{1,\infty}(\Om,\RR^d)$)
if $q<\infty$ ($q=\infty$).
\end{theorem}

With the aim of including boundary data, 
we extend Theorem \ref{thm:AF} in the following way:

\begin{theorem}
  \label{thm:relaxation}
Let $\Omega$ be a bounded domain with Lipschitz boundary and let
\begin{equation*}
  \label{eq:F}
  \mathscr{F}(u) = \left\{    \begin{array}{ll}
      \int_\Omega f(\nabla u)dx, & u \in W_h^{1,q}\\
+\infty, & u \in W^{1,q}\setminus W_h^{1,q},
    \end{array} \right.
\end{equation*}
for $1\leq q\leq \infty$, and $W^{1,q}_h$ defined as in \eqref{Wph}.
Suppose that $0\leq f(F) \leq c(1+|F|^q)$, for every $F\in \Rd$
when $q<\infty$ and $f$ is locally bounded if $q=\infty$. Then the 
seq. w.l.s.c. envelope of $\mathscr F$ in $W^{1,q}$ (weak* for $q=\infty$) is
\begin{equation*}
  \label{eq:Freldef}
\overline{  \mathscr{F}}(u) := \left\{    \begin{array}{ll}
      \int_\Omega f^{qc}(\nabla u)dx, & u \in W_h^{1,q}\\
+\infty, & u \in W^{1,q}\setminus W_h^{1,q}.
    \end{array} \right.
\end{equation*}
\end{theorem}

\begin{proof}
First assume $q\in[1,\infty)$.
By \cite{AcerbiFusco}, 
we know that $\overline{\mathscr F}$ is seq. w.l.s.c. in $W^{1,q}$
and $\overline{\mathscr F}\leq \mathscr F$.
Thus, for any $u_j \rightharpoonup u$ in $W^{1,q}$,
\begin{equation*}
 \overline{\mathscr F}(u) \leq \liminf_{j\to\infty}
\overline{\mathscr F}(u_j) \leq \liminf_{j\to\infty} \mathscr F(u_j).
\end{equation*}
It remains to prove that there exists $u_j\rightharpoonup u$ in $W^{1,q}$ such that
$$\overline{\mathscr F}(u)\geq\limsup_{j\to\infty}\mathscr F(u_j).$$ 
For simplicity, we prove the case where $\Omega = B_r$ is the
ball of radius $r$, centered at the origin. 
By Theorem~\ref{thm:AF}, for any $u\in W_h^{1,q}$ there exists a sequence
$\{w_j\}\in W^{1,q}(\Omega,\RR^d)$ such that
\begin{equation*}
  \label{eq:Ibar}
\overline{ \mathscr F}(u) = \lim_{j\to\infty}\int_{\Om}f(\na w_j)dx.
\end{equation*}
We need to modify the sequence $\{w_j\}$ to account for the boundary data. 
In order to do this, for $0<s<s+\ep<r$, let $\vph$ be a cut-off between $B_s$ and $B_{s+\ep}$, i.e.
$\vph \geq 0$, $\vph \equiv 1$ on $B_s$ and $\vph \equiv 0$ on $B_{s+\ep}$,
and $|\nabla\vph|\leq 1/\ep$. In what follows, the parameters $s$ and $\ep$ are chosen to depend on $j$. Define
\begin{equation*}
  u_j:= \vph w_j + (1-\vph)u,
\end{equation*}
so that $\{u_j \}\subset W_h^{1,q}$ and $u_j \rightharpoonup u$. Then,
\begin{align}
  \int_{B_r}f(\nabla u_j) &= \int_{B_s}f(\nabla w_j) +\int_{B_r\setminus B_{s+\ep}}
f(\nabla u) +  \int_{B_{s+\ep}\setminus B_s}f(\nabla u_j)\nonumber \\
& \leq \int_{B_s}f(\nabla w_j) +\int_{B_r\setminus B_{s+\ep}}f(\nabla u)
+  c\int_{B_{s+\ep}\setminus B_s}\left( 1+|\nabla u_j|^q\right).\label{qgrowth}
\end{align}
From the definition of $u_j$, 
$\nabla u_j = (w_j-u)\otimes \nabla \vph + (\nabla w_j - \nabla u) \vph$,
so that
\begin{equation*}
  |\nabla u_j|^q \leq c\left( \frac{1}{\ep^q}|w_j-u|^q + |\nabla w_j|^q + |\nabla u|^q \right),
\end{equation*}
and in turn
\begin{align}
  \int_{B_r}f(\nabla u_j) \leq& \int_{B_r}f(\nabla w_j) + 
\int_{B_r\setminus B_{s+\ep}}f(\nabla u)+\nonumber \\
&  c \int_{B_{s+\ep}\setminus B_s}\left(1+|\nabla w_j|^q +|\nabla u|^q\right) 
+  c \int_{B_{s+\ep}\setminus B_s}\frac{1}{\ep^q}|w_j-u|^q .\label{est1}
\end{align}
Since, up to a subsequence,
\[
\int_{B_r}|w_j-u|^q \leq \frac{1}{j^{2q+1}},
\]
choosing $\ep = \ep_j = j^{-2}$, then
\begin{equation*}
 \int_{B_{s+\ep_j}\setminus B_s}\frac{1}{\ep_j^q}|w_j-u|^q \leq \frac1j .
\end{equation*}
The choice of this subsequence, as well as of $\ep_j$, will become clear later.
Regarding the second-to-last term in \eqref{est1}, note that 
$\exists c>0$ such that
\[
\int_{B_r}\left(1+|\nabla w_j|^q +|\nabla u|^q\right) \leq c, 
\qquad \mbox{for all } j.
\]
In particular, 
\[
\sum_{i=0}^{j-1}\int_{B_{r-\frac{i}{j^2}}\setminus B_{r-\frac{i+1}{j^2}}}
\left(1+|\nabla w_j|^q +|\nabla u|^q\right)
\leq c,
\]
therefore, for every $j$ there exists $i_j$ such that 
\begin{equation*}
  \label{Sk}
\int_{B_{r-\frac{i_j}{j^2}}\setminus B_{r-\frac{i_j+1}{j^2}}} \left(1+|\nabla w_j|^q +|\nabla u|^q\right) \leq \frac cj.
\end{equation*}
Choosing $s_j = r-\frac{i_j+1}{j^2}$, then
\begin{equation}
\label{eq:eksk}
 s_j + \ep_j = r - \frac{i_j}{j^2}
\end{equation}
and \eqref{est1} becomes
\begin{equation*}
  \int_{B_r}f(\nabla u_j) \leq \int_{B_r}f(\nabla w_j) + 
\int_{B_r\setminus B_{s_j+\ep_j}}f(\nabla u)+ \frac cj.
\end{equation*}
Since $0\leq i_j\leq j-1$ for each $j$, then $i_j/j^2 \to 0$ and in turn from \eqref{eq:eksk}
\[
\left| B_r\setminus B_{s_j + \ep_j}\right|\to 0 ,
\qquad {\rm as}\,\, j\to \infty.
\]
Hence,
\begin{equation*}
\limsup_{j\to\infty}\mathscr F(u_j) =  \limsup_{j\to\infty}\int_{B_r}f(\nabla u_j) \leq
  \lim_{j\to\infty}\int_{B_r}f(\nabla w_j) =\overline{\mathscr F}(u).
\end{equation*}
This completes the proof for $q<\infty$. 

When $q=\infty$ one can argue in the same way but 
we do not have the inequality in \eqref{qgrowth}.
However, $\{u_j\}$ is bounded in $W^{1,\infty}_h$ and 
the sequence $\{f(\nabla u_j)\}$ is bounded in $L^{\infty}(\Omega)$.
Then, with the
same choices for $\ep_j$ and $s_j$, we obtain that
\[
\int_{B_{s_j+\ep_j}\setminus B_{s_j}}f(\nabla u_j)\leq c
\left| B_{s_j+\ep_j}\setminus B_{s_j}\right|\to 0.
\]
\end{proof}

\subsection{Quasiconvexity on linear strains}\label{subsection:quasi}
We recall that, for $1\leq q\leq\infty$, a locally bounded and 
Borel measurable function $f:\RR^{m\times n} \to \RR$ 
is \emph{$W^{1,q}$-quasiconvex} if 
\begin{equation*}
  f(F) \leq \fint_U f(F+\nabla \varphi)dx,\qquad \mbox{for all } \varphi
  \in W^{1,q}_0(U; \RR^m),\,\, F \in \RR^{m\times n},
\end{equation*}
where $U \subset \RR^n$ is bounded and open, and 
$\fint_U := \frac{1}{|U|}\int_U$. This definition is independent of the choice of $U$.
As it is common in the literature, we 
refer to $W^{1,\infty}$-quasiconvexity as \emph{quasiconvexity}.
The \emph{quasiconvexification} 
of a function $f:\RR^{m\times n} \to \RR$ is defined as
\begin{equation}
\label{eq:qset}
  f^{qc}:= \sup\{g:\RR^{m\times n} \to \RR \,: \, 
  g\leq f, \, g \, \mbox{is quasiconvex}\},
\end{equation}
with the convention that $f^{qc} \equiv -\infty$ if the above set is empty. 
Recall that for every locally bounded and Borel measurable $f:\RR^{m\times n} \to \RR$,
\begin{equation}
\label{eq:dac}
f^{qc}(\xi) = \inf_{\varphi \in W_0^{1,\infty}(U,\RR^m)}  \fint_{U}
f(\xi + \nabla \varphi) \, dx,
\end{equation}
whenever $U \subset \RR^n$ is a bounded and open set with $|\partial U|=0$. 
For a proof of this characterization, in the case where the set in \eqref{eq:qset} is nonempty, 
see \cite[Theorem 6.9]{Dacorogna}). 
Otherwise, we refer the reader to \cite[Theorem 4.5]{Muller_notes}.

A function $f :\Rds \to \RR$, is \emph{quasiconvex on linear
strains} if for every $E\in \Rds$
\begin{equation*}
  f(E) \leq \fint_U f(E+e(\varphi))dx,\qquad\mbox{for every }\varphi
  \in W^{1,\infty}_0(U, \RR^d),
\end{equation*}
where $U\subset \RR^d$ is open and bounded, 
and $e(\varphi) := {\rm sym} (\na \varphi)$. 
We recall that, as for quasiconvexity, this 
definition does not depend on $U$.
For  $f :\Rds \to \RR$, the \emph{quasiconvexification
on linear strains} is defined as
\begin{equation*}
  f^{qce}:= \sup\{g:\Rds \to \RR \, : \, g\leq f, \, g \,
  \mbox{is quasiconvex on linear strains}\}.
\end{equation*}
In what follows we collect some known results 
concerning the notion of quasiconvexity on linear strains.

\begin{lemma}
\label{lem:qcqce}
Let $V:\Rds \to \RR$ be locally bounded and Borel measurable, and
define $f:\Rd\to \RR$ by $f(F): = V({\rm sym}\,F)$.
Then $f$ is quasiconvex if and only if $V$
is quasiconvex on linear strains. 
\end{lemma}
\begin{proof}
The proof follows immediately from the definitions.
\end{proof}
\begin{proposition}
\label{prop:qcqce}
Under the hypotheses of Lemma \ref{lem:qcqce}, 
$f^{qc}(F) = V^{qce}({\rm sym}\,F)$, for every $F\in \Rd$.
\end{proposition}
\begin{proof} 
By Lemma \ref{lem:qcqce}, $F\mapsto V^{qce}({\rm sym}\,F)$ is quasiconvex.
Also, $V^{qce}({\rm sym}\,F)\leq V({\rm sym}\,F) = f(F)$, so that by definition
$V^{qce}({\rm sym}\,F)\leq f^{qc}(F)$ for every $F\in\Rd$. 
Now, fix $U\subset \RR^d$ open and bounded and define
\begin{equation}
  \label{eq:h}
  h(E):=\inf_{\varphi \in W^{1,\infty}_0(U,\RR^d)} \fint_U V(E+e(\varphi))dx,\qquad E\in\Rds.
\end{equation}
Then by \eqref{eq:dac}
\begin{equation*}
  h({\rm sym}\, F) = \inf_{\varphi \in W^{1,\infty}_0(U,\RR^d)} \fint_U f(F+\nabla\varphi)dx 
= f^{qc}(F).
\end{equation*}
Therefore, by Lemma \ref{lem:qcqce}, $h$ is quasiconvex on linear strains.
Taking $\varphi \equiv 0$ in \eqref{eq:h},
$h \leq V$ and then $h\leq V^{qce}$. Hence,
\begin{equation*}
 f^{qc}(F) = h({\rm sym}\, F)\leq V^{qce}({\rm sym}\, F),\qquad\mbox{for every } F\in \Rd.
\end{equation*}
This concludes the proof.
\end{proof}

As an immediate consequence of Proposition~\ref{prop:qcqce}, we get the following corollary.

\begin{corollary}
  \label{cor:dac2}
 Let $V:\Rds \to \RR$ be locally bounded and Borel measurable.
Then for all $E\in \Rds$
\begin{equation*}
  V^{qce}(E) = \inf_{\varphi \in W^{1,\infty}_0 (U,\RR^d)} \fint_U V(E+e(\varphi))dx,
\end{equation*}
where $U\subset \RR^d$ is open and bounded. 
\end{corollary}


\bibliographystyle{plain}

\end{document}